\documentclass[onefignum,onetabnum]{siamart250211}
\usepackage{braket,amsfonts}

\usepackage{array}

\usepackage[caption=false]{subfig}
\usepackage{pgfplots}
\newsiamthm{claim}{Claim}
\newsiamremark{remark}{Remark}
\newsiamremark{assumption}{Assumption}
\newsiamremark{hypothesis}{Hypothesis}
\usepackage{algorithmic}
\usepackage{graphicx,epstopdf}
\usepackage{latexsym,amssymb,amsmath,amsfonts}
\usepackage{amsopn}

\newcommand{\R}{{\mathbb R}}
\newcommand{\Z}{{\mathbb Z}}
\newcommand{\N}{{\mathbb N}}
\newcommand{\CO}{{\mathbb C}}
\newcommand{\Sp}{{\mathbb S}}

\newcommand{\no}{\nonumber}
\newcommand{\be}{\begin{eqnarray}}
\newcommand{\ben}{\begin{eqnarray*}}
\newcommand{\en}{\end{eqnarray}}
\newcommand{\enn}{\end{eqnarray*}}

\newcommand{\ov}{\overline}

\newcommand{\vep}{\varepsilon}

\newcommand{\om}{\omega}

\newcommand{\ka}{\kappa}
\newcommand{\al}{\alpha}

\newcommand{\la}{\lambda}

\newcommand{\x}{\times}

\newcommand{\re}{\textrm{Re }}

\usepackage{bm}

\usepackage{xspace}
\usepackage{bold-extra}
\usepackage[most]{tcolorbox}

\colorlet{texcscolor}{blue!50!black}
\colorlet{texemcolor}{red!70!black}
\colorlet{texpreamble}{red!70!black}
\colorlet{codebackground}{black!25!white!25}

\lstdefinestyle{siamlatex}{%
	style=tcblatex,
	texcsstyle=*\color{texcscolor},
	texcsstyle=[2]\color{texemcolor},
	keywordstyle=[2]\color{texemcolor},
	moretexcs={cref,Cref,maketitle,mathcal,text,headers,email,url},
}

\tcbset{%
	colframe=black!75!white!75,
	coltitle=white,
	colback=codebackground, % bottom/left side
	colbacklower=white, % top/right side
	fonttitle=\bfseries,
	arc=0pt,outer arc=0pt,
	top=1pt,bottom=1pt,left=1mm,right=1mm,middle=1mm,boxsep=1mm,
	leftrule=0.3mm,rightrule=0.3mm,toprule=0.3mm,bottomrule=0.3mm,
	listing options={style=siamlatex}
}

\newtcblisting[use counter=example]{example}[2][]{%
	title={Example~\thetcbcounter: #2},#1}

\newtcbinputlisting[use counter=example]{\examplefile}[3][]{%
	title={Example~\thetcbcounter: #2},listing file={#3},#1}

\DeclareTotalTCBox{\code}{ v O{} }
{ %fontupper=\ttfamily\color{texemcolor},
	fontupper=\ttfamily\color{black},
	nobeforeafter,
	tcbox raise base,
	colback=codebackground,colframe=white,
	top=0pt,bottom=0pt,left=0mm,right=0mm,
	leftrule=0pt,rightrule=0pt,toprule=0mm,bottomrule=0mm,
	boxsep=0.5mm,
	#2}{#1}

\patchcmd\newpage{\vfil}{}{}{}
\begin{tcbverbatimwrite}{tmp_\jobname_header.tex}
\title{Convergent and Efficient Iteratively Regularized Contrast Source Inversion-Type Methods for Inverse Medium Scattering Problems}
	
\author{Qiao Hu\thanks{NCMIS and Academy of Mathematics and Systems Science, Chinese Academy of Sciences, Beijing 100190, China
and School of Mathematical Sciences, University of Chinese Academy of Sciences, Beijing 100049, China (\email{huqiao2020@amss.ac.cn}).}
\and
Bo Zhang\thanks{SKLMS and Academy of Mathematics and Systems Science, Chinese Academy of Sciences, Beijing 100190,
China and School of Mathematical Sciences, University of Chinese Academy of Sciences, Beijing 100049,
China (\email{b.zhang@amt.ac.cn}).}
\and
Haiwen Zhang\thanks{SKLMS and Academy of Mathematics and Systems Science, Chinese Academy of Sciences, Beijing 100190,
China (\email{zhanghaiwen@amss.ac.cn}).}
}
	
\headers{IRCSI-type methods for inverse medium scattering}{Qiao Hu, Bo Zhang and Haiwen Zhang}
\end{tcbverbatimwrite}
\input{tmp_\jobname_header.tex}
\begin{document}

\maketitle
	
%% ------------------------------------------------------------------
%% ABSTRACT
%% ------------------------------------------------------------------
\begin{tcbverbatimwrite}{tmp_\jobname_abstract.tex}
\begin{abstract}
The contrast source inversion (CSI) method and the subspace-based optimization method (SOM) are first proposed in 1997 and 2009, respectively, and subsequently modified. The two methods and their variants share several properties and thus are called the CSI-type methods. The CSI-type methods are efficient and popular methods for solving inverse medium scattering problems, but their rigorous convergence remains an open problem. In this paper, we propose two iteratively regularized CSI-type (IRCSI-type) methods with a novel $\ell_1$ proximal term as the iteratively regularized term: the iteratively regularized CSI (IRCSI) method and the iteratively regularized 
SOM (IRSOM) method, which have a similar computation complexity to the original CSI and SOM methods, respectively,
and prove their global convergence under natural and weak conditions on the original objective function. To the best of our knowledge, this is the first convergence result for iterative methods of solving nonlinear inverse scattering problems with a fixed frequency. 
The convergence and performance of the two IRCSI-type algorithms are illustrated by numerical experiments.
\end{abstract}


\end{tcbverbatimwrite}
\input{tmp_\jobname_abstract.tex}
%% ------------------------------------------------------------------
%% END HEADER
%% ------------------------------------------------------------------

\section{Introduction}\label{sec1}

This paper is concerned with the inverse problem of scattering of time-harmonic acoustic plane waves
by an inhomogeneous medium. Such a problem has wide and important applications in many areas such
as geophysical exploration, nondestructive testing, and medical imaging.

The inverse medium scattering problem consists in recovering the refractive index of an inhomogeneous medium
(or, equivalently, the contrast function of the refractive index of the inhomogeneous medium) from the near-field
or far-field measurement data. Consequently, the inverse medium scattering problem can be formulated as the
following nonlinear operator equation
\begin{equation}\label{e1-1}
F(x)=y,
\end{equation}
where $F$ denotes the nonlinear forward operator and $x$ and $y$ denote the unknown refractive index of the
inhomogeneous medium (or its contrast) and the measurement data, respectively. The reader is referred to the
monographs \cite{C2018,C2019} for a comprehensive discussion on the mathematical and computational aspects
of inverse medium scattering problems and other types of inverse scattering problems.

Many iterative regularization methods have been developed for solving inverse medium scattering problems based
on the operator equation \eqref{e1-1} and the Fr\'{e}chet derivative (or the adjoint of the Fr\'{e}chet derivative)
of the forward operator $F$;
see, e.g., the Newton--Kantorovich method \cite{R1981} for the electromagnetic wave case,
the recursive linearization method with respect to the frequency for the case of multi-frequency measurements
or with respect to the spatial frequency for the case of fixed frequency
measurements \cite{BaoLi2004,Bao2005,BaoLi2005,BaoLi2009,BaoLiu2003,BaoLi2015,Chen1997,Zhang2014},
the preconditioned Newton method \cite{H2001,H2006,HL2010,L2010} and
the iteratively regularized Gauss--Newton method (IRGNM) with a learned projector \cite{Li2024}.
Note that these iterative methods need to solve the forward scattering problem repeatedly
in order to compute the Fr\'echet derivative (or the adjoint of the Fr\'echet derivative) of the forward operator and thus are computationally expensive.

To avoid the computations of the Fr\'echet derivative (or the adjoint of the Fr\'echet derivative) of the forward operator, much effort has been made for
developing efficient iterative methods for inverse medium scattering problems; see, e.g., the Born iterative
method (BIM) \cite{W1989}, the distorted-Born iterative method (DBIM) \cite{C1990,R2000,Y2017}, the contrast
source inversion (CSI) method \cite{V1997,V1999,V2001} and
the subspace-based optimization method (SOM) \cite{C2009,C2010,C2010b}.
In particular, the original CSI method was first proposed by van den Berg and Kleinman in 1997 \cite{V1997}.
The main idea of the CSI method is to reformulate the inverse medium scattering problem \eqref{e1-1} into an
optimization problem which is then solved by alternatively updating the contrast and the contrast source.
The SOM method was first proposed by Chen in 2009 \cite{C2009} to accelerate the CSI method and subsequently
further improved \cite{C2010,C2010b}.
The main idea of the SOM method is to use the singular value decomposition (SVD) to reduce the computation time
when updating the contrast source. As shown in \cite{C2010,V1997}, the CSI and SOM methods are
very similar in structure and can be reformulated in the following general framework
\begin{equation}\label{e1-2}
\mathop{\arg\min}_{(x, y)\in\CO^{n_1+\dots+n_J}\x\CO^{m}} \Psi(x,y)=\sum_{j=1}^J \Psi_j(x_j,y).
\end{equation}
Therefore, for convenience, the CSI and SOM methods are called the CSI-type method in this paper.
Due to the easy implementation and low computational cost of the CSI-type methods, they are widely used in the
area of electromagnetic inversion and imaging and in applications
(see, e.g., \cite{L2009,L2022,L2011,S2024,Sun2017,Sun2024,Z2009} and the monographs \cite{C2018,V2021}).
However, their convergence property has not been established yet.

We now comment on the convergence property of the iterative methods for nonlinear inverse scattering problems.
In \cite{Bao2010}, under some reasonable assumptions on the forward operator $F$ in \eqref{e1-1},
Bao and Triki established convergence with error estimates of the recursive linearization method
for solving inverse medium scattering problems with multi-frequency measurements;
by using this result, they proved in \cite{Bao2010} the first convergence result along with the error
estimate for the recursive linearization method of Chen \cite{Chen1997}.
The convergence of the recursive linearization method has also been proved in \cite{Sini2012,Sini2015}
for the inverse acoustic obstacle scattering problem with multi-frequency measurements in the 
sound-soft obstacle case. 
Convergence of iterative methods has been studied extensively for general nonlinear ill-posed problems
with the form \eqref{e1-1} under some source conditions and nonlinear conditions on the operator $F$;
see, e.g., \cite{B1992,B1997,H1997,K2010,H2007} for the iteratively regularized Gauss--Newton method (IRGNM),
\cite{Hanke97} for the Levenberg--Marquardt (LM) method, \cite{Hanke97b} for the Newton--CG method
and \cite{H1995} for the Landweber iteration method; see also the survey paper \cite{BKN2015} and
the monograph \cite{KNS08} for a comprehensive discussion on the source conditions and nonlinear conditions on $F$.
However, for the case of nonlinear inverse medium and obstacle scattering problems at a fixed frequency,
although some progress has been made through the work of Hohage \cite{H1997,H1998_NA,HW13} and others on
logarithmic or variational source conditions in the convergence analysis of Newton-type methods for the operator $F$,
the nonlinear conditions on the forward operator $F$ needed for the general convergence results could not be
verified yet (see, e.g., \cite{Hett1998} for the Landweber iteration method for inverse obstacle scattering,
\cite{H1997,H1998_NA,H1998} for Newton-type methods for inverse obstacle scattering and \cite{H2001}
for IRGNM for inverse medium scattering). Therefore, to the best of our knowledge, no convergence result has been
established for Newton-type methods in nonlinear inverse scattering problems with single-frequency measurements
(see also the survey paper \cite[pp. 788; Open Problem 3, pp. 802]{Colton2018}).

It is worth mentioning that, for both inverse acoustic and electromagnetic medium scattering, Hohage and Weidling rigorously proved the error estimates between the solution obtained by Tikhonov regularization and the true solution with respect to the noise level, under Sobolev smoothness assumptions on the refractive index \cite{H2015,H2017}.
However, their work does not consider the convergence of iterative algorithms for solving the corresponding Tikhonov regularization problems.

In this paper, we propose iteratively regularized CSI (IRCSI) and iteratively regularized SOM (IRSOM) methods,
called IRCSI-type methods, for solving the inverse medium scattering problem,
based on the general framework \eqref{e1-2} and prove their global convergence under very weak and natural
conditions on the objective function $\Psi$ (see Assumption \ref{ass1} below).
Efficient iterative methods have been developed for solving the optimization problem of the form \eqref{e1-2},
such as the Block Coordinate Descent (BCD) method \cite{T2001}, the Proximal Alternating Minimization (PAM)
method \cite{A2010} and the Proximal Alternating Linearized Minimization (PALM) method \cite{B2014}.
In the BCD method \cite{T2001}, only convergence of the subsequences is established in the nonconvex and nonsmooth setting, while the global convergence of this method is not proven.
It should be noted that both PAM and PALM adopted the standard squared $\ell_2$-norm as the proximal term and the global convergence of these two methods were established under the Kurdyka-\L{}ojasiewicz (KL) property assumption as well as a boundedness assumption of the iterative sequence \cite{A2010,B2014,Chang18}. However, for the inverse medium scattering problem, it is challenging to develop a guaranteed-convergent iterative algorithm based on this standard proximal term  in the absence of a boundedness assumption on the corresponding iterative sequence.
In order to overcome this difficulty, our IRCSI-type methods employ a novel $\ell_1$ proximal term as the iteratively regularized term, ensuring the global convergence of the iteration sequence generated by the IRCSI-type methods (see Theorem \ref{th1} below). Further, it is proved that the limit point of the iteration sequence of the IRCSI-type methods is an $\vep$-stationary point of the original objective function $\Psi$, where $\vep$ is a very small constant defined in terms of the regularization parameters (see \eqref{e5-5} in Theorem \ref{th1} below). 
As corollaries of the general convergence Theorem \ref{th1}, the convergence of the iterative sequences generated by the IRCSI and IRSOM methods follows easily (see Corollaries \ref{c1} and \ref{c2} below).
As far as we know, this is the first convergence result for iterative methods for solving nonlinear inverse
medium or obstacle scattering problems with single-frequency measurement data.
As demonstrated in the algorithmic formulations in Section \ref{sec4}, our IRCSI-type methods  have 
a similar computation complexity to the corresponding original CSI-type methods.
See Section \ref{sec4plus} for detailed discussions on the original CSI-type algorithms and our IRCSI-type algorithms.
Moreover, for the case of noisy data, we derive an estimate of the gradient of the objective function $\Psi$ at the true solution for each of our IRCSI-type methods (see \eqref{e5-20} and \eqref{eq1}).
Based on these estimates, we give a selection strategy for the regularization parameters involved in our algorithms (see \eqref{e5-21}).
Further, the numerical results presented in Section \ref{sec6} show that our IRCSI-type methods outperform
the corresponding original CSI-type methods in stability and convergence. In fact, from the numerical experiments
it seems that, in the case with noisy measurement data, the original CSI-type methods may not converge,
whilst our IRCSI-type methods converge.
In particular, in the numerical examples in Section \ref{sec6}, our algorithms for noisy data converge within only several dozen iterative steps and can provide satisfactory reconstruction results if the regularization parameters are chosen according to the selection strategy mentioned above.

The remaining part of the paper is organized as follows. In Section \ref{sec2}, we introduce the direct and inverse medium scattering problems. The IRCSI and IRSOM algorithms are proposed in Section \ref{sec4}.
In Section \ref{sec:explicit-solution}, we give the proofs of Theorems \ref{thm:stepsize}, \ref{thm:contrast},
\ref{thm:stepsize'} and \ref{thm:contrast'}, which present closed-form solutions for four subproblems associated with the IRCSI and IRSOM algorithms.
The convergence of the IRCSI and IRSOM algorithms is proved in Section \ref{sec5}.
Numerical experiments are carried out in Section \ref{sec6}
to compare the performance of the IRCSI-type methods with the original CSI-type methods.
Some concluding remarks are given in Section \ref{sec7}.

\section{The direct and inverse medium scattering problems}\label{sec2}

In this section, we describe the direct and inverse scattering problems for time-harmonic acoustic
wave propagation in an inhomogeneous medium in $\R^n$ ($n=2,3$). Assume that the inhomogeneous medium is
characterized by a piecewise smooth complex refractive index $b(x)$ with the contrast $m(x):=b(x)-1$
having a compact support in $\R^n$. The direct problem is to find the total field $u(x)$, given a wave
number $\ka>0$ and an incident wave $u^i(x)$ (in this paper we consider the incident plane wave
$u^i(x)=e^{i\ka x\cdot d}$ with $d$ being the incident direction in $\Sp^{n-1}:=\{x\in\R^n:|x|=1\}$).
The total field $u(x):=u^i(x)+u^s(x)$, which is the sum of the incident field $u^i(x)$ and the scattered
field $u^s(x)$, satisfies the following scattering problem
\begin{align}\label{e1'}
&\Delta u + \ka^2 b(x)u=0 \quad &&\mathrm{in}\;\;\mathbb{R}^n,\\ \label{e2'}
&\lim\limits_{|x|\rightarrow\infty}|x|^\frac{n-1}{2}\left(\frac{\partial u^s}{\partial |x|}-i\ka u^s\right)=0
\quad &&\mathrm{uniformly~in~all~directions~} \hat{x}:=\frac{x}{|x|}.
\end{align}
Here, \eqref{e1'} is called the reduced wave equation which is reduced to the Helmholtz equation in the case
$b(x)\equiv 1$, and \eqref{e2'} is called the Sommerfeld radiation condition.

The existence of a unique solution to the scattering problem \eqref{e1'}--\eqref{e2'} can be established by
using its equivalent Lippmann--Schwinger integral equation (see, e.g., \cite[Section 8.2]{C2019} and \cite{V2000}):
\begin{equation}\label{e2-4}
u(x) = u^i(x)+\ka^n\int_{\R^n}\Phi(\ka|x-y|)m(y)u(y)dy,\quad x\in\mathbb{R}^n.
\end{equation}
Here, $\Phi(\ka r)$ is given by
\begin{equation*}
\Phi(\ka r) = \left\{
\begin{aligned}
\frac{i}{4} H_0^{(1)}(\ka r) & , & n=2, \\
\frac{1}{4\pi} \frac{e^{i\ka r}}{\ka r} & , & n=3,
\end{aligned}
\right.
\end{equation*}
where $H_0^{(1)}$ is the Hankel function of the first kind of order zero.
It was proved in \cite[Section 8.2]{C2019} that the Lippmann--Schwinger integral equation \eqref{e2-4}
is uniquely solvable and
\begin{equation}\label{e2-5}
u=\left[I-\mathcal{T}_{\kappa,m}\right]^{-1} u^i,
\end{equation}
where the integral operator $\mathcal{T}_{\kappa,m}: C(\bar B)\rightarrow C(\bar B)$, with $B$ being a bounded open ball
containing the support of $m(x)$, is defined by
\ben
(\mathcal{T}_{\kappa,m} v)(x):=\kappa^n\int_{\R^n}\Phi(\ka|x-y|)m(y)v(y)dy, \quad x\in\bar B,
\enn
and $I-\mathcal{T}_{\kappa,m}:C(\bar B)\rightarrow C(\bar B)$ is bijective and has a bounded inverse.

The scattered field $u^s(x)$ has the following far-field asymptotic behavior
\ben
u^s(x)=\frac{e^{i\ka|x|}}{|x|^{(n-1)/2}} \left\{u^\infty(\hat{x})
+\mathcal{O}\left(\frac{1}{|x|}\right)\right\},\quad |x|\rightarrow\infty,
\enn
where $\hat{x}:=x/|x|\in\Sp^{n-1}$ denotes the unit vector in the direction of $x$.
This, together with \eqref{e2-4} and the asymptotic behavior at the infinity of $\Phi(\ka |x-y|)$ (see, e.g., \cite[(2.15) and (3.105)]{C2019}), implies that the far-field pattern $u^\infty(\hat{x})$ is given by
\be\label{far-field}
u^\infty(\hat{x}) = \theta_{n,\kappa} \int_{\R^n}e^{-i\ka\hat{x}\cdot y}m(y)u(y)dy, \qquad
\theta_{n,\kappa} = \left\{
\begin{aligned}
	\frac{\kappa^\frac{3}{2} e^{i\frac{\pi}{4}}}{\sqrt{8\pi}} & , & n=2, \\
	\frac{\kappa^2}{4\pi} & , & n=3.
\end{aligned}
\right.
\en

Making use of a number of incident plane waves $u^i_j(x)=e^{i\ka x\cdot d_j}$ with different incident
directions $d_j$, $j=1,\dots,J$, we have
\be\label{e1}
u_j(x)&=&u_j^i(x)+\ka^n\int_{\R^n}\Phi(\ka|x-y|)m(y)u_j(y)dy,\\ \label{e2}
u^\infty_j(\hat{x}) &=& \theta_{n,\kappa} \int_{\R^n}e^{-i\ka\hat{x}\cdot y}m(y)u_j(y)dy
\en
with $j=1,\ldots,J.$ These equations are called the state and data equations, respectively.

{\bf Inverse problem:} The inverse medium scattering problem is to determine the unknown contrast $m(x)$
from the given measurement data $u^\infty_j(\hat{x})$, $j=1,\dots,J$.

For each incident plane wave $u^i_j$ ($j=1,\dots,J$), we introduce the far-field operator $\mathcal{F}_j: L^2(B) \to L^2(\mathbb{S}^{n-1})$, which maps the contrast $m$ to its corresponding far-field pattern $u_j^\infty$:
\be\label{eq:farfied}
\mathcal{F}_j(m) = u_j^\infty.
\en
It easily follows from the equations \eqref{e2-5} and \eqref{e2} that
\be\label{eq:forward_opera}
\mathcal{F}_j(m)(\hat{x}) = \theta_{n,\kappa} \int_{B} e^{-i\kappa \hat{x} \cdot y} m(y) \left( \left( I -\mathcal{T}_{\kappa,m} \right)^{-1} u_j^i \right)(y) dy.
\en
It is straightforward to observe that each $\mathcal{F}_j$ is a strongly nonlinear operator with respect to $m$. From this and the analyticity of the exponential kernel $e^{-i\kappa \hat{x} \cdot y}$ in $\hat{x} \in \mathbb{S}^{n-1}$, we can see that the above inverse problem is strongly nonlinear and severely ill-posed.

In \eqref{e1} and \eqref{e2}, both the total field $u_j(x)$ and the contrast $m(x)$ are unknown. For each $j$,
define the product $m(x)u_j(x)$ as a single quantity $\om_j(x)$, which is referred to as the contrast source:
\begin{equation*}
\om_j(x) := m(x)u_j(x).
\end{equation*}
By multiplying both sides of the equation \eqref{e1} with $m(x)$, the state equation becomes
\be\label{e3}
\om_j(x) = m(x)u_j^i(x) + m(x)\ka^n\int_{\R^n}\Phi(\ka|x-y|)\om_j(y)dy. 
\en
Note that the data equation \eqref{e2} becomes
\be\label{e7}
u^\infty_j(\hat{x}) &=& \theta_{n,\kappa} \int_{\R^n}e^{-i\ka\hat{x}\cdot y}\om_j(y)dy.
\en

For the numerical implementation of the inverse problem, we need to discretize the equations \eqref{e3}
and \eqref{e7}. In this paper, we use the method proposed in \cite{V2000} for discretization.
We divide the entire space $\R^n$ into open boxes with a given small number $h>0$:
\ben
B_{p,h}:=\{x=(x_1,\dots,x_n)\in\R^n:(p_k-\frac{1}{2})h < x_k < (p_k+\frac{1}{2})h,\;\; k=1,\dots,n\},
\enn
where $p=(p_1,\dots,p_n)\in\Z^n$ and $ph$ represents the center point of $B_{p,h}$. Since the support of
$m(x)$ is assumed to be compact, we can choose a sufficiently large $N\in\N$ and an open bounded set
$G\subset\R^n$ (independent of $h$) such that
\ben
\textrm{supp}(m)\subset \bigcup_{p\in\Z^n_N}\overline{B}_{p,h}\subset G,
\enn
where
\ben
\Z^n_N:=\{p=(p_1,\dots,p_n)\in\Z^n: -\frac{N}{2}<p_k\leqslant\frac{N}{2}, k=1,\dots,n\}.
\enn
Further, define
\ben
\Phi_{p,h}:=\left\{
\begin{aligned}
\Phi(\ka|p|h) & , && 0\neq p \in\Z^n_N, \\
0 & , && p=0.
\end{aligned}
\right.
\enn
Then the state equation \eqref{e3} is discretized as the discrete system: for $j=1,\dots,J$,
\be\label{e4}
\om_{j,p,h}=m_{p,h}\cdot u^i_j(ph) + h^n m_{p,h}\cdot \ka^n\sum_{k\in\Z^n_N}\Phi_{p-k,h}\cdot \om_{j,k,h},\;\; p\in\Z^n_N,
\en
where $m_{p,h}$ is the value of the contrast $m(x)$ at $ph$ as in \cite{V2000}.
We sort the points of $\Z^n_N$ by $p^1,\dots,p^M$, with $M=N^n$, in a fixed order, e.g.,
the lexicographic order, and rewrite the equation \eqref{e4} as the discrete state equation
\begin{equation}\label{e5}
\bm\om_{j}= \bm m \odot \bm u^i_{j} + \bm m\odot T\bm\om_{j},\;\;\;j=1,\dots,J,
\end{equation}
where
\begin{align*}
\bm m &:= (m_{p^1,h},\dots,m_{p^M,h})^\top\in\CO^M,\\
\bm u^i_{j} &:= (u^i_j(p^1h),\dots,u^i_j(p^Mh))^\top\in\CO^M, \\
\bm \om_{j} &:= (\om_{j,p^1,h},\dots,\om_{j,p^M,h})^\top\in\CO^M,
\end{align*}
\begin{equation*}
T := \ka^n h^n
\begin{pmatrix}
\Phi_{p^1-p^1,h} & \cdots & \Phi_{p^1-p^M,h}\\
\vdots & \ddots & \vdots \\
\Phi_{p^M-p^1,h} & \cdots & \Phi_{p^M-p^M,h}
\end{pmatrix}
\in\CO^{M\x M},
\end{equation*}
and the operation $\odot$ denotes the element--wise multiplication. Define $D(\bm m):=diag(\bm m)$
as a diagonal matrix and denote by $I_M$ the identity matrix of order $M$.
Then \eqref{e5} can be rewritten in the matrix form:
\be\label{e5'}
[I_M - D(\bm m) T] \bm \om_{j} = D(\bm m) \bm u^i_{j},\;\;\;j=1,\dots,J.
\en

\begin{remark}\label{r1}
Through the paper, we assume that $m(x)$ is a piecewise smooth function with a compact support and satisfies the condition mentioned in \cite[Theorem 2.1]{V2000}.
Further, assume that the homogeneous integral equation corresponding to \eqref{e1} possesses only the trivial solution.
Then, it has been shown in \cite[Theorem 2.1]{V2000} that for each $j=1,\dots,J$ and for any sufficiently small $h>0$, the discrete system \cite[Equation (8)]{V2000}:
\begin{equation}\label{e4'}
u_{j,p,h} = u^i_j(ph) + h^n \sum_{k\in\Z^n_N}\Phi_{p-k,h} m_{k,h} u_{j,k,h}, \qquad p\in\Z^n_N,
\end{equation}
or
\ben
[I_M-TD(\bm m)]\bm u_{j}=\bm u^i_{j}\;\;\mathrm{~with~}\bm u_{j}:=(u_{j,p^1,h},\dots,u_{j,p^M,h})^\top\in\CO^M
\enn
is uniquely solvable, and the error between $u_j(ph)$ (the solution of \eqref{e1}) and $u_{j,p,h}$
(the solution of \eqref{e4'}) is $O(h^2(1+|\ln h|))$. Thus, the matrix $[I_M-D(\bm m)T]=[I_M-TD(\bm m)]^\top$
in \eqref{e5'} is invertible, and the error between $\om_j(ph)$ (the solution of \eqref{e3}) and $\om_{j,p,h}$
(the solution of \eqref{e4}) is also $O(h^2(1+|\ln h|))$.
\end{remark}

We now discretize 
the data equation \eqref{e7}.
Given $Q\in\N^+$,
we choose different
observation directions
$\hat{x}^1,\dots,\hat{x}^Q$. Then, for each $j=1,\dots,J$, the data equation \eqref{e7} can be rewritten as
\begin{equation}\label{e6}
\bm u^\infty_{j} = T^\infty \bm\om_{j}+\bm{\varepsilon}_{j,h}\approx
T^\infty \bm\om_{j}\quad\textrm{for sufficiently small}~h>0,
\end{equation}
since the infinity norm of $\bm\varepsilon_{j,h}$ tends to $0$ as $h\rightarrow +0$ (see \cite{V93,V2000}). Here,
\ben
\bm u^\infty_{j}:=(u^\infty_j(\hat{x}^1),\dots,u^\infty_j(\hat{x}^Q))^\top\in\CO^Q
\enn
and
\ben
T^\infty:=\theta_{n,\kappa} h^n
\begin{pmatrix}
	e^{-i\ka\hat{x}^1\cdot p^1h} & \cdots & e^{-i\ka\hat{x}^1\cdot p^Mh}\\
	\vdots & \ddots & \vdots \\
	e^{-i\ka\hat{x}^Q\cdot p^1h} & \cdots & e^{-i\ka\hat{x}^Q\cdot p^Mh}
\end{pmatrix} \in\CO^{Q\x M}.
\enn
Therefore, we assume that $h>0$ is  sufficiently small throughout the paper.

Let $\bm u^{\infty,\delta}_j$ denote the noisy data satisfying that
\be\label{eq2}
\|\bm u^{\infty,\delta}_j-\bm u^{\infty}_j\|\leqslant\delta,\quad j=1,\dots,J.
\en
Hereafter, $\|\cdot\|$ represents the $\ell_2$-norm of a vector $\|\cdot\|_2$ unless otherwise specified.
Then the (discrete) {\em inverse medium scattering problem} considered in this paper is as follows:
Given $\bm u^i_{j}$ and for a fixed and sufficiently small $h$,
reconstruct $\bm m$ from the noisy data $\bm u^{\infty,\delta}_{j}$, $j=1,\dots,J$.

\section{Iteratively regularized CSI and SOM algorithms}\label{sec4}
In this section, we propose two novel algorithms, iteratively regularized CSI (IRCSI) method and iteratively regularized SOM (IRSOM) method, for solving the inverse medium scattering problem. Our IRCSI and IRSOM algorithms are modified variants of the classical CSI \cite{V1997,V1999,V2001} and SOM \cite{C2009,C2010} methods, respectively.
Our methods incorporate a novel $\ell_1$ proximal term as the iteratively regularized term, while retaining a similar computational complexity to the corresponding original CSI-type methods. 
In Sections \ref{sec4-1} and \ref{sec4-2}, we present the IRCSI and IRSOM algorithms, respectively.
Detailed discussions on the original CSI-type algorithms and our IRCSI-type algorithms are given in Section \ref{sec4plus}.
Throughout this paper, let $\langle\cdot,\cdot\rangle$ denote the inner product of complex vectors and let $(\bm \xi)_\ell$ represent the $\ell$-th element of a vector $\bm \xi$.

\subsection{IRCSI algorithm}\label{sec4-1}
The IRCSI method updates $\bm\om_j$ and $\bm m$ alternatively to minimize
the following objective function, which is a weighted sum of the errors between the state
equation \eqref{e5} and the data equation \eqref{e6} with $j=1,\dots,J$:
\be\label{e3-1}
&&F^{csi}(\bm\om_{1},\dots,\bm\om_{J},\bm m)\\ \no
&&\qquad\;=\sum_{j=1}^{J}\eta_{s,j}\|\bm m\odot\bm u^i_{j}
+ \bm m \odot T\bm\omega_{j} - \bm\omega_{j}\|^2
+\sum_{j=1}^{J}\eta_{d,j}\|\bm u^{\infty,\delta}_{j}-T^\infty\bm\omega_{j}\|^2.
\en
Here and throughout the paper, if not stated otherwise, we set the weights $\eta_{s,j}$ and $\eta_{d,j}$ to be positive constants.
The IRCSI method consists of constructing the sequences $\{\bm \om^{(r)}_j\}_{r\in\N}$ and
$\{\bm m^{(r)}\}_{r\in\N}$
for updating $\bm{\om}_j$ and $\bm{m}$, respectively,
as follows.

Choose  $\gamma,\beta\geq 0$ to be constants. By letting $\bm\om,\bm\om^{(r)}\in\CO^{MJ}$ represent the variables $(\bm\om_1,\dots,\bm\om_J)$, $(\bm\om^{(r)}_1,\dots,\bm\om^{(r)}_J)$, respectively, the updating process of the IRCSI method can be written as
\begin{align}
\bm\om^{(r+1)}&\in\mathop{\arg\min}_{\bm{\om}}{F}^{csi}(\bm\om,\bm m^{(r)})+\gamma\|\bm\om-\bm\om^{(r)}\|_1, \label{e4-10}\\
\bm m^{(r+1)}&\in\mathop{\arg\min}_{\bm m}{F}^{csi}(\bm\om^{(r+1)},\bm m)+\beta\|\bm m-\bm m^{(r)}\|_1. \label{e4-11}
\end{align}
Here and throughout the paper, if not stated otherwise, $\|\cdot\|_p$ represents the $\ell_p$-norm of a vector for any $1\leq p\leq\infty$. 

Write $F^{csi}(\bm\om_{1},\dots,\bm\om_{J},\bm m)=\sum_{j=1}^{J}F_j^{csi}(\bm\om_j,\bm m)$ with
\ben
F_j^{csi}(\bm\om_j,\bm m):=\eta_{s,j}\|\bm m\odot\bm u^i_{j}+\bm m \odot T\bm\omega_{j}-\bm\omega_{j}\|^2
+\eta_{d,j}\|\bm u^{\infty,\delta}_{j}-T^\infty\bm\omega_{j}\|^2.
\enn
Then the subproblem \eqref{e4-10} can be solved approximately by the one-step Polak--Ribi\`ere conjugate gradient (PR-CG) method:
\begin{align}\label{e3-9}
&\bm\om_j^{(r+1)}=\bm\om_j^{(r)} + s_j^{(r)}\bm v_j^{(r)},\quad j=1,\dots,J,\\ \label{e3-10}
&s_j^{(r)}\in\mathop{\arg\min}_{s}F_j^{csi}(\bm\om_j^{(r)}+s\bm v_j^{(r)},\bm m^{(r)})+\gamma\|s\bm v_j^{(r)}\|_1.
\end{align}
Here, the update direction $\bm v_j^{(r)}$ is the PR-CG direction:
\begin{equation}\label{e4-8}
\bm v_j^{(r)}=\left\{
\begin{aligned}
&g_j^{(0)}, &&\quad r=0,\\
&g_j^{(r)}+\frac{\re\langle g_j^{(r)},g_j^{(r)}-g_j^{(r-1)}\rangle}{\|g_j^{(r-1)}\|^2}\bm v_j^{(r-1)}, &&\quad r\ge1
\mathrm{~and~} g_j^{(r-1)}\neq 0, \\
&g_j^{(r)}, &&\quad r\ge 1 \mathrm{~and~} g_j^{(r-1)}=0,
\end{aligned}
\right.
\end{equation}
with the gradient $g_j^{(r)}=\nabla_{\bm\om_j}F_j^{csi}$ evaluated at $\bm\om_j^{(r)}$ and $\bm m^{(r)}$ (see Definition \ref{def1}). 
The following theorem shows that the problem \eqref{e3-10} has a closed-form solution.
Thus we set the step size $s_j^{(r)}$ to be given by the formula \eqref{e4-16} in Theorem \ref{thm:stepsize}. 
The proof of Theorem \ref{thm:stepsize} will be given in Section \ref{sec:proof-ircsi}.

\begin{theorem}[Optimality of step size]\label{thm:stepsize}
Let $\gamma\geq 0$ and $\eta_{s,j},\eta_{d,j}>0$.
Then the problem \eqref{e3-10} has a minimizer $s_j^{(r)}$ given explicitly by
\begin{equation}\label{e4-16}
s_j^{(r)}=\begin{cases}
0, & \|\bm m^{(r)}\odot T\bm v^{(r)}_j-\bm v^{(r)}_j\|=\|T^\infty\bm v^{(r)}_{j}\|=0,\\[4pt]
S_\tau(z), & \text{otherwise},
\end{cases}
\end{equation}
where $\tau$ and $z$ are given by
\ben
\tau &=&\frac{\gamma\|\bm v_j^{(r)}\|_1}{2\eta_{s,j}\|\bm v^{(r)}_{j}-\bm m^{(r)}\odot T\bm v^{(r)}_{j}\|^2
  + 2\eta_{d,j}\|T^\infty \bm v^{(r)}_{j}\|^2},\\
z &=& \frac{-\langle g_j^{(r)},\bm v_j^{(r)}\rangle}{2\eta_{s,j}\|\bm v^{(r)}_{j}-\bm m^{(r)}\odot T\bm v^{(r)}_{j}\|^2
  + 2\eta_{d,j}\|T^\infty\bm v^{(r)}_{j}\|^2},
\enn
and $S_\tau(\cdot):\mathbb{C}\to\mathbb{C}$ is a complex soft‑thresholding function defined by
\begin{equation}\label{eq7}
S_\tau(a):=\begin{cases}
(|a|-\tau)\dfrac{a}{|a|}, & |a|>\tau,\\[4pt]
0, & |a|\le\tau.
\end{cases}
\end{equation}
\end{theorem}

The subproblem \eqref{e4-11} is a special LASSO problem which can be solved exactly. 
To be more specific, let $\bm u_j^{(r+1)} = \bm u^i_j + T\bm\om_j^{(r+1)}$ with $\bm\om_j^{(r+1)}$ given by \eqref{e3-9}. 
Then the following theorem shows that the problem \eqref{e4-11} also has a closed-form solution.
Hence,
we choose $\bm m^{(r+1)}$
to be given by the formula \eqref{e4-17} in Theorem \ref{thm:contrast}.
The proof of Theorem \ref{thm:contrast} will be given in Section \ref{sec:proof-ircsi}.

\begin{theorem}[Optimality of $\bm m^{(r+1)}$]\label{thm:contrast}
Let $\beta\geq 0$ and $\eta_{s,j}>0$. Then the problem \eqref{e4-11} has a minimizer $\bm m^{(r+1)}$ with its $\ell$-th element given by
\be\label{e4-17}
(\bm m^{(r+1)})_\ell=
\begin{cases}
(\bm m^{(r)})_\ell, & (\bm u^{(r+1)}_{j})_\ell = 0 \text{ for all } j,\\
(\bm m^{(r)})_\ell + S_\tau \bigl(z - (\bm m^{(r)})_\ell\bigr), & \text{otherwise},
\end{cases}
\en
where
$\tau$ and $z$ are given by
\ben
\tau=\frac{\beta}{2\sum_{j=1}^J \eta_{s,j}|(\bm u_{j}^{(r+1)})_\ell|^2}, \quad
z=\frac{\sum_{j=1}^J \eta_{s,j}\ov{(\bm u_{j}^{(r+1)})_\ell} (\bm\om_j^{(r+1)})_\ell}
{\sum_{j=1}^J\eta_{s,j}|(\bm u^{(r+1)}_{j})_\ell|^2},
\enn
and $S_\tau(\cdot)$ is defined as in \eqref{eq7}.
\end{theorem}

To sum up, the proposed iteratively regularized CSI (IRCSI) algorithm consists of \eqref{e3-9}, \eqref{e4-8},
\eqref{e4-16} and \eqref{e4-17} and is presented in Algorithm \ref{alg1}.

It should be remarked that an $\ell_1$ proximal term was also used in \cite{G2020} for regularized variational
image restoration models with smoothly truncated regularizer functions.
It should be remarked further that regularized CSI methods have also been proposed in \cite{V1999} with the
regularizing term being the standard TV term and in \cite{ZL2017} with the regularization term being the Bregman
distance induced by a general uniformly convex functional.

\begin{algorithm}[htbp]
\caption{The IRCSI algorithm}
\label{alg1}
\noindent\textbf{Input:} \parbox[t]{\dimexpr\linewidth - 5em}{%
  wave number $\ka$, linear operators $T$ and $T^\infty$,\\
  positive constants $\gamma,\beta,\eta_{s,j},\eta_{d,j}$, $j=1,\dots,J$,\\
  incident wave $\bm u^i_j$, measurement data $\bm u^{\infty,\delta}_j$, $j=1,\dots,J$.
}

\begin{algorithmic}[1]
\STATE \textbf{Initialize:} $r=0$ and $(\bm\om_1^{(0)},\dots,\bm\om_J^{(0)},\bm m^{(0)})$ is given by back-propagation \cite{V1997,V1999,V2001}.
\WHILE {(the termination criterion is not satisfied)}
  \STATE For all $j=1,\dots,J$, compute the gradient $g_j^{(r)}=\nabla_{\bm\om_j}F^{csi}(\bm\om^{(r)},\bm m^{(r)})$ and determine the PR-CG search direction by \eqref{e4-8}.
  \STATE Update $\bm\om_j^{(r+1)}$ as $\bm\om_j^{(r+1)}=\bm\om_j^{(r)}+s_{j}^{(r)}\bm v_j^{(r)}$ with $s_{j}^{(r)}$ given by \eqref{e4-16}.
  \STATE Update $\bm m^{(r+1)}$ by \eqref{e4-17}.
  \STATE $r = r + 1$.
\ENDWHILE
\STATE \textbf{Return:} $(\bm\om_1^{(r)},\dots,\bm\om_J^{(r)},\bm m^{(r)})$.
\end{algorithmic}
\end{algorithm}

\subsection{IRSOM algorithm}\label{sec4-2}

We now present the IRSOM method. Recall the data equation \eqref{e6} and consider the
Singular Value Decomposition (SVD) of $T^\infty=\sum_k\bm u_k\la_k\bm v_k^*$ with the singular values
$\la_1\ge\cdots\ge\la_{L_0}>\la_{L_0+1}=\dots=\la_M=0$ being placed in a non-increasing order.
Then we split $\bm\om_j$ up into two parts:
\ben
\bm \om_j = \bm \om_j^s + \bm \om_j^n,
\enn
where $\bm\om_j^s$ (the signal of $\bm\om_j$) is spanned by the first $L_\al$ right singular vectors
$\bm v_k\,(k\le L_\al)$, and $\bm\om_j^n$ (the noise of $\bm\om_j$) is spanned by the remaining $M-L_\al$
right singular vectors $\bm v_k\,(k>L_\al)$. Here, $L_\al$ is a predefined integer selected in $[1,L_0]$ and $\bm\om_j^s$ is uniquely determined as follows:
\be\label{e3-8}
\bm\om_j^s = \sum_{k=1}^{L_\al}\frac{\langle\bm u_j^{\infty,\delta},\bm u_k\rangle}{\la_k}\bm v_k.
\en
Thus only $\bm\om_j^n=\al_1\bm v_{L_\al+1}+\cdots+\al_{M-L_\al}\bm v_{M}$ needs to be solved, that is,
\ben
\bm\om_j^n = \begin{pmatrix}
\bm v_{L_\al+1}, \dots, \bm v_{M}
\end{pmatrix} \bm\al_j
\enn
with an $(M-L_\al)$-dimensional vector $\bm\al_j$ to be determined.
Let $V^n\in\CO^{M\times(M-L_\al)}$ denote the matrix $(\bm v_{L_\al+1},\cdots,\bm v_{M})$. Then
\begin{equation}\label{e3-7}
\bm\om_j=\bm\om_j^s+V^n\bm\al_j, \quad j=1,\dots,J,
\end{equation}
and the objective functional of IRSOM is
\be\no
F^{som}(\bm\al_{1},\dots,\bm\al_{J},\bm m)&=&\sum_{j=1}^{J}\eta_{s,j}\|\bm m\odot\bm u^i_{j}+\bm m\odot T (\bm\om_j^s+V^n\bm\al_{j})-(\bm\om_j^s+V^n\bm\al_{j})\|^2\\ \label{e3-4}
&&+\sum_{j=1}^{J}\eta_{d,j}\|\bm u^{\infty,\delta}_{j}-T^\infty\bm\om_j^s-T^\infty V^n\bm\al_{j}\|^2.
\en
Here and throughout the paper, we set the weights $\eta_{s,j}$ and $\eta_{d,j}$ to be positive constants.
The IRSOM method constructs the sequences $\{\bm\al^{(r)}_j\}_{r\in\N}$ and
$\{\bm m^{(r)}\}_{r\in\N}$
for updating $\bm\al_j$ and $\bm{m}$, respectively,
as follows.

Choose  $\gamma,\beta\geq 0$ to be constants. By letting $\bm\al,\bm\al^{(r)}\in\CO^{(M-L_\al)J}$ represent the variables $(\bm\al_1,\dots,\bm\al_J)$, $(\bm\al^{(r)}_1,\dots,\bm\al^{(r)}_J)$, respectively,
the updating process of the IRSOM method can be written as
\begin{align}
\bm\al^{(r+1)}&\in\mathop{\arg\min}_{\bm\al}F^{som}(\bm\al,\bm m^{(r)})+\gamma\|\bm\al-\bm\al^{(r)}\|_1, \label{e4-12}\\
\bm m^{(r+1)} &\in\mathop{\arg\min}_{\bm m}F^{som}(\bm\alpha^{(r+1)},\bm m)+\beta\|\bm m-\bm m^{(r)}\|_1. \label{e4-13}
\end{align}

Write $F^{som}(\bm\al_{1},\dots,\bm\al_{J},\bm m)=\sum_{j=1}^{J}F_j^{som}(\bm\al_j,\bm m)$ with
\ben
F^{som}_j(\bm\al_j,\bm m)&:=&\eta_{s,j}\|\bm m\odot\bm u^i_{j}+\bm m\odot T(\bm\om_j^s+V^n\bm\al_{j})
-(\bm\om_j^s+V^n\bm\al_{j})\|^2 \\
&&+\eta_{d,j}\|\bm u^{\infty,\delta}_{j}-T^\infty\bm\om_j^s-T^\infty V^n\bm\al_{j}\|^2.
\enn
Similarly as in Section \ref{sec4-1}, we still use the one-step PR-CG method to solve the subproblem \eqref{e4-12} approximately, that is,
\begin{align}\label{e4-14}
&\bm\al_j^{(r+1)} =\bm\al_j^{(r)}+s_j^{(r)}\bm \rho_j^{(r)},\quad j=1,\dots,J,\\ \label{e4-15}
&s_j^{(r)} \in\mathop{\arg\min}_{s}F_j^{som}(\bm\al_j^{(r)}+s\bm\rho_j^{(r)},\bm m^{(r)})+\gamma\|s\bm\rho_j^{(r)}\|_1.
\end{align}
Here,
\begin{equation}\label{e4-20}
\bm \rho_j^{(r)}=\left\{
\begin{aligned}
&g_j^{(0)}, &&\quad r=0,\\
&g_j^{(r)}+\frac{\re\langle g_j^{(r)},g_j^{(r)}-g_j^{(r-1)}\rangle}{\|g_j^{(r-1)}\|^2}\bm\rho_j^{(r-1)},
     &&\quad r\ge 1 \mathrm{~and~} g_j^{(r-1)}\neq 0,\\
&g_j^{(r)}, &&\quad r\ge 1 \mathrm{~and~} g_j^{(r-1)}=0,
\end{aligned}
\right.
\end{equation}
with the gradient $g_j^{(r)}=\nabla_{\bm\al_j} F_j^{som}$ evaluated at $\bm\al_j^{(r)}$ and $\bm m^{(r)}$ (see Definition \ref{def1}).
We have the following theorem on an explicit solution of the problem \eqref{e4-15}.
Thus we set the optimal step size $s_j^{(r)}$ to be given by the formula \eqref{e4-18} in Theorem \ref{thm:stepsize'}.
The proof of Theorem \ref{thm:stepsize'} will be presented in Section \ref{sec:proof-irsom}.

\begin{theorem}[Optimality of step size]\label{thm:stepsize'}
Let $\gamma\geq 0$ and $\eta_{s,j},\eta_{d,j}>0$. Then the problem \eqref{e4-15} has a minimizer $s_j^{(r)}$ given by
\begin{equation}\label{e4-18}
s_j^{(r)}=\begin{cases}
0, & \|V^n\bm\rho^{(r)}_j-\bm m^{(r)}\odot TV^n\bm\rho^{(r)}_j\|=
  \|T^\infty V^n\bm\rho^{(r)}_{j}\|=0,\\[4pt]
S_\tau(z), & \text{otherwise},
\end{cases}
\end{equation}
where $\tau$ and $z$ are given by
\ben
\tau &=& \frac{\gamma\|\bm\rho_j^{(r)}\|_1}{2\eta_{s,j}\|V^n\bm\rho^{(r)}_{j}-\bm m^{(r)}\odot T V^n\bm\rho^{(r)}_{j}\|^2
  +2\eta_{d,j}\|T^\infty V^n\bm\rho^{(r)}_{j}\|^2},\\
z&=&\frac{-\langle g_j^{(r)},\bm\rho_j^{(r)}\rangle}{2\eta_{s,j}\|V^n\bm\rho^{(r)}_j-\bm m^{(r)}\odot TV^n\bm\rho^{(r)}_j\|^2
  + 2\eta_{d,j}\|T^\infty V^n\bm\rho^{(r)}_{j}\|^2},
\enn
and $S_\tau(\cdot)$ is defined as in \eqref{eq7}.
\end{theorem}

Moreover, let $\bm\om_j^{(r+1)}=\bm\om_j^s+V^n\bm\al_j^{(r+1)}$ with $\bm\al_j^{(r+1)}$ given by \eqref{e4-14} and let $\bm u_j^{(r+1)}=\bm u^i_j+T\bm\om_j^{(r+1)}$. Then the following theorem gives an explicit solution of the problem \eqref{e4-13}.
Hence,
we choose $\bm m^{(r+1)}$
to be given by the formula \eqref{e4-19} in Theorem \ref{thm:contrast'}.
The proof of Theorem \ref{thm:contrast'} is provided in Section \ref{sec:proof-irsom}.

\begin{theorem}[Optimality of $\bm m^{(r+1)}$]\label{thm:contrast'}
Let $\beta\geq 0$ and $\eta_{s,j}>0$. Then the problem \eqref{e4-13} has a minimizer $\bm m^{(r+1)}$ with its $\ell$-th element given by
\be\label{e4-19}
(\bm m^{(r+1)})_\ell=
\begin{cases}
(\bm m^{(r)})_\ell, & (\bm u^{(r+1)}_{j})_\ell = 0 \text{ for all } j,\\
(\bm m^{(r)})_\ell + S_\tau \bigl(z - (\bm m^{(r)})_\ell\bigr), & \text{otherwise},
\end{cases}
\en
where
$\tau$ and $z$ are given by
\ben
\tau=\frac{\beta}{2\sum_{j=1}^J \eta_{s,j}|(\bm u_{j}^{(r+1)})_\ell|^2}, \quad
z=\frac{\sum_{j=1}^J \eta_{s,j}\ov{(\bm u_{j}^{(r+1)})_\ell} (\bm\om_j^{(r+1)})_\ell}
{\sum_{j=1}^J\eta_{s,j}|(\bm u^{(r+1)}_{j})_\ell|^2},
\enn
and $S_\tau(\cdot)$ is defined as in \eqref{eq7}.
\end{theorem}

To sum up, the proposed iteratively regularized SOM (IRSOM) algorithm consists of \eqref{e4-14}, \eqref{e4-20},
\eqref{e4-18} and \eqref{e4-19} and is presented in Algorithm \ref{alg2}.

\begin{algorithm}
\caption{The IRSOM algorithm}
\label{alg2}
\noindent\textbf{Input:} \parbox[t]{\dimexpr\linewidth - 5em}{%
    wave number $\ka$, linear operators $T$ and $T^\infty$,\\
    positive constants $\gamma,\beta,\eta_{s,j},\eta_{{d,j}}$, $j=1,\dots,J$, \\
    incident wave $ \bm u^i_j$, measurement data $\bm u^{\infty,\delta}_j$, $j=1,\dots,J$.\\
    the vector $\bm\om_j^s$ and the matrix $V^n\in\CO^{M\x (M-L_\al)}$ given as in \eqref{e3-7}.
}

\begin{algorithmic}[1]
\RETURN $r=0$, $\bm\al_j^{(0)}=0$, and $\bm m^{(0)}$ is given by back-propagation \cite{V1997,V1999,V2001}.
\WHILE {(the termination criterion is not satisfied)}
\STATE {For all $j=1,\dots,J$, compute the gradient $g_j^{(r)}=\nabla_{\bm\al_j}F^{som}(\bm\al^{(r)},\bm m^{(r)})$
     and then determine the PR-CG search direction by \eqref{e4-20}.}
\STATE{Update $\bm\al_j^{(r+1)}$ by $\bm\al_j^{(r+1)}=\bm\al_j^{(r)}+s_{j}^{(r)}\bm\rho_j^{(r)}$ with $s_{j}^{(r)}$
      given by \eqref{e4-18}.}
\STATE {Calculate $\bm\om_j^{(r+1)}=\bm\om_j^s+V^n\bm\al_j^{(r+1)}$.}
\STATE {Update $\bm m^{(r+1)}$ by \eqref{e4-19}.}
\STATE {$r=r+1.$}
\ENDWHILE
\STATE \textbf{Return:} $(\bm\al_1^{(r)},\dots,\bm\al_J^{(r)},\bm m^{(r)})$.
\end{algorithmic}
\end{algorithm}

\subsection{Discussions on the original CSI-type algorithms and our IRCSI-type algorithms}\label{sec4plus}
First, we give a comparison between the original CSI-type algorithms and our IRCSI-type algorithms.
For the original CSI algorithm \cite{V1997}, the iterative procedure is basically identical to Algorithm \ref{alg1} with $\gamma=\beta=0$, except that the weights $\eta_{s,j}$ are regarded as $\left(\sum_{j=1}^J\|\bm m^{(r)}\odot\bm u^i_j\|^2\right)^{-1}$ in updating $\bm\om_j^{(r+1)}$
and $\bm m^{(r+1)}$ at the $r$-th iteration (see \eqref{e4-11}, \eqref{e3-9} and \eqref{e3-10}).
This means that the objective function $F^{csi}$ is always changing (since the values of $\eta_{s,j}$ are different at each iteration), making the convergence analysis of the original CSI method more complicated. It should be mentioned that the original CSI method with weights taking to be positive constants actually works well
in practical computation. As demonstrated in Section \ref{sec6}, for the original CSI method, we always set
$\eta_{s,j}=\left(\sum_{j=1}^J\|\bm m^{(0)}\odot\bm u^i_j\|^2\right)^{-1}$
with $\bm m^{(0)}$ being an initial guess of the contrast $\bm m$ (meaning that the weights remain unchanged during the updating process) and numerical experiments show that
the weights so chosen have very little impact on the reconstruction results.
Moreover, the original SOM algorithm \cite{C2010} is exactly the same as Algorithm \ref{alg2} in the reduced case $\gamma=\beta=0$.
Furthermore, both the original CSI-type methods and our IRCSI-type  methods make use of  the state equation \eqref{e3} and the data equation \eqref{e7} (rather than the far-field equation \eqref{eq:farfied}) for solving the inverse problem.
Note that the far-field operator \eqref{eq:forward_opera} is highly nonlinear with respect to $m$. However, the second term of the state equation \eqref{e3} is a bilinear operator with respect to $\omega_j$ and $m$ and the right hand side of the data equation \eqref{e7} is a linear operator with respect to $\omega_j$. Therefore, 
the original CSI-type methods and our IRCSI-type  methods effectively transform the inherent strong nonlinear inverse medium scattering problem into a series of alternatingly solvable coupled subproblems with linear and bilinear structures.

Second, we discuss the reasons for choosing the $\ell_1$ norm as the proximal term in our IRCSI-type algorithms. One reason is that 
using the $\ell_1$ proximal term can ensure high computational efficiency.
To be more specific, the subproblems \eqref{e4-11}, \eqref{e3-10}, \eqref{e4-13} and \eqref{e4-15} are special one-dimensional LASSO problems for the case $\gamma,\beta>0$ and linear least-squares problems for the case $\gamma=\beta=0$, which are involved in our IRCSI-type methods and the original CSI-type methods, respectively.
Since these subproblems have closed-form solutions (see Theorems \ref{thm:stepsize}, \ref{thm:contrast}, \ref{thm:stepsize'} and \ref{thm:contrast'}), thus
our IRCSI-type algorithms have a similar computational complexity to the corresponding original CSI-type algorithms.
Another reason is that using the $\ell_1$ proximal term can ensure that our IRCSI and IRSOM algorithms are globally convergent under natural and weak conditions on the original objective function.
This global convergent property will be proved in Section \ref{sec5}.
It should be noted that the convergence of the original CSI and SOM methods remains an open problem. In fact, by the numerical experiments in Section \ref{sec6} it seems that the original CSI and SOM methods may not converge in the case with noisy measurement data, whilst our IRCSI-type methods converge. 
For a discussion on alternative choices for the proximal term in our IRCSI-type methods, see Remark \ref{remark:ell_1}.

\section{Proofs of Theorems \ref{thm:stepsize}, \ref{thm:contrast},
\ref{thm:stepsize'} and \ref{thm:contrast'}}\label{sec:explicit-solution}
We now give the proofs of Theorems \ref{thm:stepsize}, \ref{thm:contrast},
\ref{thm:stepsize'} and \ref{thm:contrast'}, which present closed-form solutions for four subproblems associated with the IRCSI and IRSOM algorithms.

\subsection{Proofs of Theorems \ref{thm:stepsize} and \ref{thm:contrast}}\label{sec:proof-ircsi}
In this subsection, we adopt the notations used in Section \ref{sec4-1}.
\begin{proof}[Proof of Theorem \ref{thm:stepsize}] 
Define the state error by
\ben
E_{s,j}^{(r)} := \bm m^{(r)} \odot \bm u^i_{j} + \bm m^{(r)} \odot T\bm\om^{(r)}_{j} - \bm\om^{(r)}_{j}
\enn
and the data error by
\ben
E_{d,j}^{(r)} := \bm u^{\infty,\delta}_{j} - T^\infty\bm\om^{(r)}_{j}.
\enn
Then we only need to prove that the step size $s_j^{(r)}$ given by \eqref{e4-16} is a solution of the optimization problem
\begin{align}
&\mathop{\arg\min}_{s\in\mathbb{C}} \left\{ \eta_{s,j}\|E_{s,j}^{(r)}+s(\bm m^{(r)}\odot T\bm v_j^{(r)}-\bm v_j^{(r)})\|^2\right.\no\\
&\qquad\qquad\left.+ \eta_{d,j}\|E_{d,j}^{(r)}-sT^\infty\bm v_j^{(r)}\|^2 + \gamma\|\bm v_j^{(r)}\|_1|s| \right\}.\label{eq8}
\end{align}
For brevity, we drop the indices $j,r$ and introduce the notations
\[
D := \bm m^{(r)}\odot T\bm v^{(r)}_j - \bm v^{(r)}_j,\; G := T^\infty\bm v^{(r)}_j,\; E_s := E_{s,j}^{(r)},\; E_d := E_{d,j}^{(r)},\; c := \gamma\|\bm v^{(r)}_j\|_1,
\]
and $\eta_s:=\eta_{s,j}$, $\eta_d:=\eta_{d,j}$. Then the objective function for the above optimization problem is given by
\begin{equation}\label{app:obj}
f(s):=\eta_s\|E_s+s D\|^2+\eta_d\|E_d-s G\|^2+c|s|.
\end{equation}

Define the smooth part $h(s) := \eta_s\|E_s+s D\|^2 + \eta_d\|E_d-s G\|^2$. According to Definition \ref{def1}, its gradient is
\begin{equation}\label{app:grad_g}
\nabla_s h(s) = 2\bigl[\eta_s D^H(E_s+sD) - \eta_d G^H(E_d-sG)\bigr].
\end{equation}
It is clear that $f(s)=h(s)+c|s|$ is convex. Thus for the problem \eqref{eq8}, there exists at least one global minimizer and 
\be\label{eq9}
s^*\in\mathbb{C}~\text{is a global minimizer iff}~
0 \in \nabla_s h(s^*) + \partial_s (c|s|)(s^*).
\en
Hereafter, $\partial$ denotes the subdifferential of a function.
Let 
\ben
B_0 := \eta_s\|D\|^2+\eta_d\|G\|^2,\quad b_0 := \eta_s D^H E_s - \eta_d G^H E_d.
\enn
Note that if $B_0\neq 0$, then $\tau ={c}/{(2B_0)}$ and $z=-{b_0}/{B_0}$ since 
\begin{align*}
b_0 &= \eta_s D^H E_s - \eta_d G^H E_d\\ 
&= \eta_s \bm v^H (\bm m\odot T - I)^H E_s - \eta_d \bm v^H (T^\infty)^H E_d\\
&= \bm v^H \left[\eta_s (\bm m\odot T - I)^H E_s - \eta_d (T^\infty)^H E_d\right] = \frac{1}{2} \bm v^H g,
\end{align*}
where $\bm v$, $\bm m$ and $g$ denote $\bm v^{(r)}_j$, $\bm m^{(r)}$ and the gradient $g_j^{(r)}=\nabla_{\bm\om_j}F_j^{csi}(\bm\om_j^{(r)},\bm m^{(r)})$ defined by Definition \ref{def1}, respectively.
Now we consider the following two cases.

\paragraph{Case 1: $2|b_0|\leq c$ (equivalently, $B_0=0$, or $B_0\neq 0$ and $|z|\leq\tau$)}
Note that $\partial_s (c|s|)(0) = \{ca: |a|\le 1\}$. Thus we have
\[
0 \in \nabla_s h(0) + \partial_s (c|s|)(0).
\]
This, together with \eqref{eq9}, implies that $s=0$ is a global minimizer of the problem \eqref{eq8}.

\paragraph{Case 2: $2|b_0|> c$ (equivalently, $B_0\neq 0$ and $|z|>\tau$)}
Let $s^*$ be a global minimizer of the problem \eqref{eq8}.
Then it follows from \eqref{eq9} that $s^*\neq 0$. Thus using
$\nabla_s(c|s|)(s^*)=c\frac{s^*}{|s^*|}$ and the optimality condition in \eqref{eq9}, we have
\[
2B_0 |s^*|\hat{s^*} + 2b_0 + c\hat{s^*} = 0
\quad\textrm{with}~
\hat{s^*} := \frac{s^*}{|s^*|}.
\]
This, together with the fact that $B_0\neq 0$, implies that
\[
|s^*| = \frac{2|b_0| - c}{2B_0},\quad
\hat{s^*} = -\frac{b_0}{|b_0|}.
\]
Consequently, we have
\[
s^* = \frac{2|b_0|-c}{2B_0}\left(-\frac{b_0}{|b_0|}\right) = S_\tau(z).
\]

Based on the discussions in the above two cases, the proof is thus completed.
\end{proof}

\begin{proof}[Proof of Theorem \ref{thm:contrast}]
Note that for any vector, its $\ell_1$-norm and squared $\ell_2$-norm are additively separable with respect to its components.
Thus by the definition \eqref{e3-1} of $F^{csi}(\bm\om,\bm m)$, it suffices to show that for any arbitrarily fixed $\ell$,
$(\bm m^{(r+1)})_{\ell}$ defined by \eqref{e4-17} is a minimizer of the problem
\be\label{eq10}
\mathop{\arg\min}_{s\in\mathbb{C}}\sum_{j=1}^{J}\eta_{s,j}|s\cdot(\bm u_j^{(r+1)})_\ell-(\bm\om_{j}^{(r+1)})_\ell|^2
+ \beta|s-(\bm m^{(r)})_\ell|.
\en

In fact, define
\[
s_0 := (\bm m^{(r)})_\ell,\qquad 
a_j := (\bm u_j^{(r+1)})_\ell,\qquad b_j := (\bm\om_j^{(r+1)})_\ell,
\]
and $\eta_j := \eta_{s,j}$. Then the objective function for the above optimization problem is given by
\begin{equation*}
f_\ell(s) = \sum_{j=1}^J \eta_j |a_j s - b_j|^2 + \beta |s - s_0|.
\end{equation*}
Note that the function $f_\ell:\mathbb{C}\to\mathbb{R}$ is convex. Hence, for the problem \eqref{eq10}, we have
\be\label{eq11}
s^*\in\mathbb{C}~\textrm{is a global minimizer if and only if}~0\in\partial_s f_\ell(s^*).
\en
Define
\[
A := \sum_{j=1}^J \eta_j |a_j|^2,\qquad 
B := \sum_{j=1}^J \eta_j \overline{a_j}b_j.
\]
Then we have $\tau={\beta}/{(2A)}$ and $z= {B}/{A}$ if $A>0$.
The smooth part of $f_\ell$ is $h(s)=\sum_{j=1}^J\eta_j|a_j s-b_j|^2$ and its complex gradient is
\begin{equation}\label{app2:gradg}
\nabla_s h(s) = 2\sum_{j=1}^J\eta_j\bigl(|a_j|^2 s - \overline{a_j}b_j\bigr) = 2A s - 2B.
\end{equation}
For the non‑smooth term $q(s)=\beta|s-s_0|$, it is known that its complex gradient at $s\neq s_0$ and subdifferential at $s=s_0$ are given by
\be\label{eq13}
\nabla_s q(s) = \beta\frac{s-s_0}{|s-s_0|}\quad\textrm{for}~s\neq s_0,\quad 
\partial_s q(s_0) = \{\beta a: a\in\mathbb{C},\,|a|\le 1\}.
\en
Now we distinguish between the following two cases.

\paragraph*{Case 1: $|2A s_0-2B|\leq\beta$ (equivalently, $A=0$, or $A\neq 0$ and $|s_0-z|\leq\tau$)}
It follows that
\be\label{eq12}
0 \in \nabla_s h(s_0) + \partial_s q(s_0) = 2A s_0 - 2B + \{\beta a:| a|\le1\}.
\en
Thus by \eqref{eq11}, we have that $s=s_0$ is a global minimizer of the problem \eqref{eq10}.

\paragraph*{Case 2: $|2A s_0-2B|>\beta$ (equivalently, $A\neq 0$ and $|s_0-z|>\tau$)}
Let $s^*$ be a global minimizer of the problem \eqref{eq10}.
It can be seen that \eqref{eq12} does not hold and thus \eqref{eq11} implies that $s^*\neq s_0$.
Hence, by \eqref{eq11}, \eqref{app2:gradg} and \eqref{eq13}, we have
\[
0=2A d + (2A s_0 - 2B) + \beta\frac{d}{|d|} =2A\left[d+(s_0-z)+\tau\frac{d}{|d|}\right]
\]
with $d:=s^*-s_0$.
This implies that
\ben
|d| = |z-s_0| - \tau,
\quad
\frac{d}{|d|} = \frac{z-s_0}{|z-s_0|}.
\enn
Consequently,
\[
s^*=d + s_0= \bigl(|z-s_0| - \tau\bigr)\frac{z-s_0}{|z-s_0|}+s_0=s_0+S_\tau(z-s_0).
\]

Therefore, from the above two cases, we obtain the statement of this theorem.
\end{proof}

\subsection{Proofs of Theorems \ref{thm:stepsize'} and \ref{thm:contrast'}}\label{sec:proof-irsom}
In this subsection, we adopt the notations used in Section \ref{sec4-2}.
\begin{proof}[Proof of Theorem \ref{thm:stepsize'}]
Define the state error by
\ben
E_{s,j}^{(r)} := \bm m^{(r)} \odot \bm u^i_{j} + \bm m^{(r)} \odot T(\bm\om_j^s+V^n\bm\al_{j}^{(r)}) - (\bm\om_j^s+V^n\bm\al_{j}^{(r)})
\enn
and the data error by
\ben
E_{d,j}^{(r)} := \bm u^{\infty,\delta}_{j} - T^\infty(\bm\om_j^s+V^n\bm\al_{j}^{(r)}).
\enn
Then the problem \eqref{e4-15} can be rewritten as
\ben
&&\mathop{\arg\min}_{s\in\mathbb{C}} \Bigl\{\eta_{s,j}\|E_{s,j}^{(r)}+s(\bm m^{(r)}\odot T V^n\bm\rho_j^{(r)}-V^n\bm\rho_j^{(r)})\|^2\\
&&\qquad\qquad+\eta_{d,j}\|E_{d,j}^{(r)} - s T^\infty V^n\bm\rho^{(r)}_{j}\|^2 + \gamma\|\bm \rho_j^{(r)}\|_1|s| \Bigr\}.
\enn
Introduce the notations
\[
D := \bm m^{(r)}\odot T V^n\bm\rho_j^{(r)}-V^n\bm\rho_j^{(r)},\quad G := T^\infty V^n\bm\rho^{(r)}_{j},
\]
and
\[
E_s := E_{s,j}^{(r)},\quad E_d := E_{d,j}^{(r)},\quad c := \gamma\|\bm\rho^{(r)}_j\|_1, \quad \eta_s:=\eta_{s,j}, \quad \eta_d:=\eta_{d,j}.
\]
Thus the objective function for the above optimization problem is given by
\ben
f(s)=\eta_s\|E_s+s D\|^2+\eta_d\|E_d-s G\|^2+c|s|.
\enn
Therefore, the statement of this theorem can be obtained by using the same arguments as in the proof of Theorem \ref{thm:stepsize}.
\end{proof}

\begin{proof}[Proof of Theorem \ref{thm:contrast'}]
Similarly to Theorem \ref{thm:contrast}, proving the statement of this theorem is equivalent to proving that for any arbitrarily fixed $\ell$, $(\bm m^{(r+1)})_{\ell}$ defined by \eqref{e4-19}
is a minimizer of the problem \eqref{eq10} with $\eta_{s,j}$, $\beta$, $(\bm u_j^{(r+1)})_\ell$, $(\bm\om_{j}^{(r+1)})_\ell$ and $(\bm m^{(r)})_\ell$ given as in Section \ref{sec4-2}.
Therefore, we can argue in a same way as in the proof of Theorem \ref{thm:contrast} to show that the statement of this theorem holds.
\end{proof}

\section{Convergence analysis of the IRCSI and IRSOM algorithms}\label{sec5}

Throughout this section, we assume that $\gamma,\beta$ are positive constants.
The IRCSI and IRSOM algorithms can be reformulated in a unified framework. For $x_j\in\CO^{n_j},y\in\CO^m$
with integers $m,n_j>0$, $j=1,\dots,J$, write $x=(x_1,\dots,x_J)$ and define
$\Psi(x,y)=\Psi(x_1,\dots,x_J,y):=\sum_{j=1}^J\Psi_j(x_j,y)$ with
$\Psi_j(\cdot,\cdot):\CO^{n_j}\x\CO^{m}\rightarrow\R$ being continuously differentiable.
Now consider the following iteratively regularized CSI-type (IRCSI-type) method:
\begin{align}
x_j^{k+1} &= x_j^k +  s_j^k v_j^k, \quad j=1,\dots,J, \label{e5-1}\\
s_j^k &\in \mathop{\arg\min}_{s}\Psi_j(x_j^k +sv_j^k, y^k) + \gamma\|s v_j^k\|_1, \label{e5-2}\\
y^{k+1} &\in \mathop{\arg\min}_{y}\sum_{j=1}^J\Psi_j(x_j^{k+1},y) +\beta\|y-y^k\|_1, \label{e5-3}
\end{align}
where the PR-CG direction $v_j^k$ is defined by
\begin{equation}\label{e5-4}
v_j^k:=\left\{\begin{aligned}
&g_j^0, &&\quad k=0,\\
&g_j^k+\frac{\re\langle g_j^k,g_j^k-g_j^{k-1}\rangle}{\|g_j^{k-1}\|^2}v_j^{k-1},
    &&\quad k\ge 1\mathrm{~and~} g_j^{k-1}\neq 0,\\
&g_j^k, &&\quad k\ge 1 \mathrm{~and~} g_j^{k-1}=0,
\end{aligned}
\right.
\end{equation}
with $g_j^k=\nabla_{x_j}\Psi_j(x_j^k,y^k)$ which is defined in the following Definition \ref{def1}.
In the rest of the paper, let $\|\cdot\|_\infty$ denote the infinity norm of a vector or the  maximum absolute row sum norm of a matrix.

\begin{definition}\label{def1}
Let	$\psi(z):\CO^n\rightarrow\R$ be a real-valued function defined on the complex domain $\CO^n$.
For any $z=z_r+iz_i$ with $z_r,z_i\in\R^n$, the function $\psi(z)$ can be equivalently represented
as a real-variable function $\widetilde{\psi}(z_r,z_i)$ mapping $\R^n\x\R^n$ to $\R$.
$\psi(z)$ is called continuously differentiable (denoted as $\psi\in C^1$) if $\widetilde{\psi}(z_r,z_i)$
is continuously differentiable in $\R^n\x\R^n$. The complex gradient operator $\nabla_z\psi(z)$ is defined as
\ben
\nabla_z\psi(z):=\nabla_{z_r}\widetilde{\psi}(z_r,z_i)+i\nabla_{z_i}\widetilde{\psi}(z_r,z_i).
\enn
\end{definition}

If $y=\bm m$, $x_j=\bm\om_j$ and $\Psi_j(x_j,y)=F_j^{csi}(\bm\om_j,\bm m)$, $j=1,\dots,J$, then the IRCSI-type
method \eqref{e5-1}--\eqref{e5-3} becomes the IRCSI method with the subproblems \eqref{e5-2} and \eqref{e5-3}
having explicit solutions \eqref{e4-16} and \eqref{e4-17}, respectively.
If $y=\bm m$, $x_j=\bm\al_j$ and $\Psi_j(x_j,y)=F_j^{som}(\bm\al_j,\bm m)$, $j=1,\dots,J$, then the IRCSI-type
method \eqref{e5-1}--\eqref{e5-3} becomes the IRSOM method with the subproblems \eqref{e5-2} and \eqref{e5-3}
having explicit solutions \eqref{e4-18} and \eqref{e4-19}, respectively.

To prove the convergence of the iterative solutions $\{(x^k,y^k)\}_{k\in\N}$ of the IRCSI-type method
\eqref{e5-1}--\eqref{e5-3} with $x^k=(x^k_1,\dots,x^k_J)$, we need to make certain
assumptions on the objective function $\Psi(x,y)=\sum_{j=1}^J\Psi_j(x_j,y)$.

\begin{assumption}\label{ass1}
(i) For each $j=1,\dots,J$, $\Psi_j(x_j,y):\CO^{n_j}\x\CO^{m}\rightarrow\R$ is $C^1$, that is,
$\nabla_{x_j}\Psi_j(x_j,y)$ and $\nabla_{y}\Psi_j(x_j,y)$ are both continuous.
	
(ii) $\inf_{(x,y)\in\CO^{n_1+\dots+n_J}\x\CO^{m}}\Psi(x,y)>-\infty$.
\end{assumption}

\begin{remark}\label{r2}
It is easy to prove that $\nabla_{\bm\om_j}F_j^{csi}$, $\nabla_{\bm m}F_j^{csi}$, $\nabla_{\bm\al_j}F_j^{som}$
and $\nabla_{\bm m}F_j^{som}$ are continuous for all $j$, so $F_j^{csi}$ and $F_j^{som}$ are $C^1$
functions for all $j$. In addition, it is clear that $F^{csi}\ge 0$ and $F^{som}\ge 0$.
Thus the functions $F_j^{csi}$, $F_j^{som}$, $F^{csi}$ and $F^{som}$ satisfy Assumption \ref{ass1}.
\end{remark}

\begin{theorem}\label{th1}
Suppose Assumption \ref{ass1} holds and let $\{z^k=(x^k, y^k)\}_{k\in\N}$ be the sequence generated by the IRCSI-type
method \eqref{e5-1}--\eqref{e5-3} with $\gamma,\beta>0$. Then the following statements hold:
	
$(i)$ The sequence $\{z^k\}_{k\in\N}$ is convergent.
	
$(ii)$ Let $L\geq\max_{1\leq j\leq J} \sqrt{n_j}$ be a constant such that $\|\xi\|_1\le L\|\xi\|$ for any $\xi\in\CO^{n_j}$, $j=1,\dots,J$
and let $\varepsilon=\max\{\gamma L,\beta\}$. Then the limit point $z^*$ of $\{z^k\}_{k\in\N}$ is an
$\varepsilon$-stationary point of the objective function $\Psi$, that is,
\begin{equation}\label{e5-5}
\|\nabla\Psi(z^*)\|_\infty\le\varepsilon.
\end{equation}
\end{theorem}

\begin{proof}
(i) From \eqref{e5-1}--\eqref{e5-3} it follows that for any $k\in\N$,
\begin{align}
\gamma\|x_j^{k+1}-x_j^k\|_1 &\leqslant \Psi_j(x_j^k,y^k) - \Psi_j(x_j^{k+1},y^k),\;\quad j=1,\dots,J, \label{e5-6}\\
\beta\|y^{k+1}-y^k\|_1 &\leqslant \Psi(x^{k+1},y^k) - \Psi(x^{k+1},y^{k+1}). \label{e5-7}
\end{align}
Summing up \eqref{e5-6} with respective to $j$ gives
\begin{equation}\label{e5-8}
\gamma\|x^{k+1}-x^k\|_1 \leqslant \Psi(x^k,y^k) - \Psi(x^{k+1},y^k).
\end{equation}
By \eqref{e5-7} and \eqref{e5-8}, we know $\Psi(x^{k},y^{k})\ge\Psi(x^{k+1},y^{k+1})$. This, together with
the fact that $\inf\Psi>-\infty$, implies that
\begin{equation}\label{e5-9}
\lim\limits_{k\rightarrow\infty}\Psi(z^k) = \Psi^*\;\;\mathrm{~for~some~}\;\Psi^*\in\R.
\end{equation}
Now, summing up \eqref{e5-7} and \eqref{e5-8} for all $k\in\N$, we get
\begin{align*}
\sum_{k=0}^\infty\left[\gamma\|x^{k+1}-x^{k}\|_1 + \beta\|y^{k+1}-y^k\|_1\right]
&\le \sum_{k=0}^\infty \left[\Psi(x^{k},y^k) - \Psi(x^{k+1},y^{k+1})\right] \\
&\le \Psi(z^0) - \Psi^* < +\infty,
\end{align*}
which means that the sequence $\{z^k\}_{k\in\N}$ is a Cauchy sequence so $\{z^k\}_{k\in\N}$ is convergent.
	
(ii) Since $s_j^k$ is a minimum solution with respect to $s$ of $\Psi_j(x_j^k+s v_j^k,y^k)+\gamma\|s v_j^k\|_1$
(see \eqref{e5-2}), then the partial derivative (see Definition \ref{def1})
\ben
\partial_s\Psi_j(x_j^k+s v_j^k,y^k)|_{s=s_j^k}=\langle\nabla_{x_j}\Psi_j(x_j^{k+1},y^k),v_j^k\rangle
\enn
is a sub-gradient of $-\gamma\|v_j^k\|_1\partial|s_j^k|$, and thus we have
\begin{equation}\label{e5-10}
|\langle\nabla_{x_j}\Psi_j(x_j^{k+1},y^k),v_j^k\rangle|\le \gamma\|v_j^k\|_1.
\end{equation}
Let $(x^*,y^*)$ be the convergent point of $\{(x^k,y^k)\}_{k\in\N}$. Since $\Psi_j$ is a $C^1$ function, then
\begin{equation}\label{e5-11}
\lim\limits_{k\rightarrow\infty}\nabla_{x_j}\Psi_j(x_j^{k+1},y^k)=\nabla_{x_j}\Psi_j(x_j^*,y^*)=:g_j^*,
\quad\forall j=1,\dots,J.
\end{equation}
We now claim that for all $j=1,\dots,J$,
\begin{equation}\label{e5-12}
\|g_j^*\|\le \gamma L.
\end{equation}
In fact, suppose this were not true, that is, $\|g_{j_0}^*\|>\gamma L$ for some $j_0\in\{1,\cdots,J\}.$
Then there would exist a sufficiently large $k_0$ such that $\|g_{j_0}^{k-1}\|>\gamma L$ for all $k>k_0$.
By \eqref{e5-4} we obtain that
\begin{equation}\label{e5-13}
\|v_{j_0}^k-g_{j_0}^k\|\le\frac{\|g_{j_0}^{k}\|\|g_{j_0}^k-g_{j_0}^{k-1}\|}{\|g_{j_0}^{k-1}\|^2}\|v_{j_0}^{k-1}\|
\le\frac{\|g_{j_0}^{k}\|\|g_{j_0}^k-g_{j_0}^{k-1}\|}{(\gamma L)^2}\|v_{j_0}^{k-1}\|, \quad\forall k>k_0.
\end{equation}
Since $\lim\limits_{k\rightarrow\infty}g_j^k = g_j^*$ for $j=1,\cdots,J$, then there is a $k_1>k_0$ such that
\ben
\|v_{j_0}^k\|\le\|g_{j_0}^*\|+1+\frac{1}{2}\|v_{j_0}^{k-1}\|,\quad\forall k>k_1.
\enn
This yields that $\|v_{j_0}^k\|\le 2\|g_{j_0}^*\|+2+\|v_{j_0}^{k_1}\|/2$ for all $k>k_1$. Thus, by letting $k\rightarrow\infty$
in \eqref{e5-13} we get
\begin{equation}\label{e5-14}
\lim\limits_{k\rightarrow\infty} v_{j_0}^k =g_{j_0}^*.
\end{equation}
Letting $k\rightarrow\infty$ in \eqref{e5-10} and making use of \eqref{e5-11} and \eqref{e5-14}, we have
\ben
\|g_{j_0}^*\|^2\le\gamma\|g_{j_0}^*\|_1\le\gamma L\|g_{j_0}^*\|,
\enn
where use has been made of the assumption $\|\xi\|_1\le L\|\xi\|$ for any $\xi\in\CO^{n_j}$, $j=1,\dots,J$.
This contradicts the condition that $\|g_{j_0}^*\|>\gamma L$, so the claim \eqref{e5-12} is true.
Thus for all $j$,
\begin{equation*}
\|g_j^*\|_\infty \le\|g_j^*\|\le\gamma L.
\end{equation*}
Since $\|\nabla_{x}\Psi(x,y)\|_\infty=\max_{1\le j\le J}{\|\nabla_{x_j}\Psi_j(x_j,y)\|_\infty}$, then
\begin{equation}\label{e5-15}
\|\nabla_{x}\Psi(x^*,y^*)\|_\infty \le \gamma L.
\end{equation}
	
Similarly, by \eqref{e5-3} we know that $\nabla_y\Psi(x^{k+1},y^{k+1})$ is a sub-gradient of
$-\beta\partial_y\|y-y^k\|_1$ at the point $y^{k+1}$, so
\begin{equation*}
\|\nabla_y\Psi(x^{k+1},y^{k+1})\|_\infty\le\beta, \quad \forall k.
\end{equation*}
Letting $k\to\infty$ in the above inequality and noting that $\Psi_j$ is a $C^1$ function, we get
\begin{equation}\label{e5-16}
\|\nabla_y \Psi(x^*,y^*)\|_\infty\le\beta.
\end{equation}
Combining \eqref{e5-15} and \eqref{e5-16} gives that
\ben
\|\nabla\Psi(z^*)\|_\infty\le\max\{\gamma L,\beta\}.
\enn
The proof is thus complete.
\end{proof}

This theorem establishes the global convergence of the iteration sequence $\{z^k=(x^k,y^k)\}_{k=1}^\infty$
generated by the IRCSI-type method \eqref{e5-1}--\eqref{e5-3} with the limit $z^*=(x^*,y^*)$ of the
sequence $\{z^k\}_{k=1}^\infty$ satisfying $\|\nabla\Psi(z^*)\|_\infty\le\varepsilon$, where
$\varepsilon=\max\{\gamma L,\beta\}$.
Such $z^*$ is called an $\vep$-stationary point of the objective function $\Psi(x,y)$ (see also \cite{G2020}
for this definition for regularized variational image restoration problems).

In applications, the measurement data $\bm u_j^{\infty,\delta}$ is usually noisy, so that it is actually
impossible to find an exact critical point $z^*=(x^*,y^*)$ of the objective function $\Psi(x,y)$
(i.e., $\nabla\Psi(x^*,y^*)=0$). In fact, in this case, it is sufficient to find an $\vep$-stationary
point of the objective function $\Psi(x,y)$ for certain small $\varepsilon$ depending actually on the noise level.
In this sense, the global convergence result presented in Theorem \ref{th1} is reasonable, as demonstrated by the numerical
experiments in Section \ref{sec7}. Further, the termination criterion can be chosen as
\begin{equation}\label{e5-17}
\|\nabla\Psi(x^k,y^k)\|_\infty\le 2\varepsilon.
\end{equation}

If $y=\bm m$, $x_j=\bm\om_j$ and $\Psi_j(x_j,y)=F_j^{csi}(\bm\om_j,\bm m)$, $j=1,\dots,J$ or if $y=\bm m$,
$x_j=\bm\al_j$ and $\Psi_j(x_j,y)=F_j^{som}(\bm\al_j,\bm m)$, $j=1,\dots,J$, then Assumption \ref{ass1} is satisfied,
as indicated by Remark \ref{r2}. Then the following corollaries follow easily from Theorem \ref{th1}.

\begin{corollary} 
\label{c1}
Let $\{z^k=(\bm\om^{(k)},\bm m^{(k)})\}_{k\in\N}$ be the sequence generated by the IRCSI algorithm
(Algorithm \ref{alg1}) with $\gamma,\beta>0$. Then the following statements hold.
	
(i) The sequence $\{z^k\}_{k\in\N}$ is convergent.
	
(ii) Let $L\geq\max_{1\leq j\leq J} \sqrt{n_j}$ be a constant such that $\|\xi\|_1\le L\|\xi\|$ for any $\xi\in\CO^{n_j}$, $j=1,\dots,J$
and $\varepsilon=\max\{\gamma L,\beta\}$. Then the limit point $z^*$ of $\{z^k\}_{k\in\N}$ is an
$\vep$-stationary point of the objective function $F^{csi}$, that is,
\begin{equation*}
\|\nabla F^{csi}(z^*)\|_\infty\le\varepsilon.
\end{equation*}
\end{corollary}

\begin{corollary} 
\label{c2}
Let $\{z^k=(\bm\al^{(k)},\bm m^{(k)})\}_{k\in\N}$ be the sequence generated by the IRSOM algorithm
(Algorithm \ref{alg2}) with $\gamma,\beta>0$. Then the following statements hold.
	
(i) The sequence $\{z^k\}_{k\in\N}$ is convergent.
	
(ii) Let $L\geq\max_{1\leq j\leq J} \sqrt{n_j}$ be a constant such that $\|\xi\|_1\le L\|\xi\|$ for any $\xi\in\CO^{n_j}$, $j=1,\dots,J$
and $\varepsilon=\max\{\gamma L,\beta\}$. Then the limit point $z^*$ of $\{z^k\}_{k\in\N}$ is an
$\vep$-stationary point of the objective function $F^{som}$, that is,
\begin{equation*}
\|\nabla F^{som}(z^*)\|_\infty\le\varepsilon.
\end{equation*}
\end{corollary}

\begin{remark}\label{remark:ell_1}
We now discuss alternative choices for the proximal term in our IRCSI-type method \eqref{e5-1}--\eqref{e5-4}.
One possible choice is the standard squared $\ell_2$-norm, that is, we use $\|s v_j^k\|^2$ instead of $\|s v_j^k\|_1$ in \eqref{e5-2} and $\|y- y^k\|^2$ instead of $\|y-y^k\|_1$ in \eqref{e5-3}, respectively.
Note that the standard squared $\ell_2$-norm $\|y- y^k\|^2$ has been employed in both PAM \cite{A2010} and PALM \cite{B2014} method as the proximal term.
However, in these methods,
the global convergence is established only under the Kurdyka-\L{}ojasiewicz (KL) property assumption together with a boundedness assumption of the iterative sequence \cite{A2010,B2014,Chang18}.
If we select the standard squared $\ell_2$-norm for the proximal term in our IRCSI-type method, it is unclear whether the corresponding iterative sequence $\{(x^k,y^k)\}_{k\in\N}$ is bounded. Therefore, the convergence of the resulting algorithm is unknown to us and cannot be directly established using the results in \cite{A2010,B2014}.
Another possible choice is the $\ell_p$-norm with $1<p\leq\infty$, that is, we use
$\|s v_j^k\|_p$ and $\|y- y^k\|_p$ as the proximal terms in the subproblems \eqref{e5-2} and \eqref{e5-3}, respectively.
With this choice, using similar arguments as in the proof of Theorem \ref{th1}, we can also prove that the iterative sequence $\{(x^k,y^k)\}_{k\in\N}$ generated by the corresponding variant of the IRCSI-type method is convergent.
However, for the modified IRCSI and IRSOM methods with such a choice, the corresponding subproblems \eqref{e4-11} and \eqref{e4-13} with the proximal term $\|\bm m-\bm m^{(r)}\|_1$ replaced by $\|\bm m-\bm m^{(r)}\|_p$ 
($1<p\leq\infty$) usually do not have closed-form solutions,
which leads to a significant increase in computational cost.
\end{remark}

In the remainder of this section, we discuss the choice of the regularization parameters $\gamma$ and $\beta$.
Corollaries \ref{c1} and \ref{c2} ensure that the iterative sequence
$\{z^k=(\bm\om^{(k)},\bm m^{(k)})\}_{k\in\N}$
obtained by the IRCSI algorithm (or $\{z^k=(\bm\al^{(k)},\bm m^{(k)})\}_{k\in\N}$ obtained
by the IRSOM algorithm) converges globally to an $\vep$-stationary point $z^*$ of
the objective function $F^{csi}(z)$ with $z=(\bm\om,\bm m)$ (or $F^{som}(z)$ with $z=(\bm\al,\bm m)$).
Rewrite $F^{csi}((\bm\om,\bm m);\bm u^{\infty,\delta})$ for $F^{csi}(\bm\om,\bm m)$ defined
in \eqref{e3-1} and $F^{som}((\bm\al,\bm m);\bm u^{\infty,\delta})$ for $F^{som}(\bm\al,\bm m)$ defined
in \eqref{e3-4} to indicate explicitly the dependence on $\bm u^{\infty,\delta}=(\bm u^{\infty,\delta}_1,\ldots,\bm u^{\infty,\delta}_J)$.
Then, by Corollaries \ref{c1} and \ref{c2} we have
\be\label{e5-18}
\;\;\|\nabla F^{csi}(z^*;\bm u^{\infty,\delta})\|_\infty\le\vep\qquad\;
(\mbox{or}\;\;\|\nabla F^{som}(z^*;\bm u^{\infty,\delta})\|_\infty\le\vep),
\en
where $\varepsilon=\max\{\gamma L,\beta\}$, as defined in Corollary \ref{c1} (or Corollary \ref{c2}).

On the other hand, let $\bm m^\dagger$ be the ground truth
solution of the discrete inverse medium scattering problem defined at the end of Section \ref{sec2}
and let $\bm\om^\dagger=(\bm\om^\dagger_1,\ldots,\bm\om^\dagger_J)$ with $\bm\om^\dagger_j$ being the solution of the discrete state equation (\ref{e5}) (that is,
the discrete form of the Lippmann--Schwinger integral equation (\ref{e3})) corresponding to $\bm m^\dagger$.
Further, let $\bm\alpha^\dagger=(\bm\alpha^\dagger_1,\ldots,\bm\alpha^\dagger_J)=((V^n)^*\bm\omega^\dagger_1,\ldots,(V^n)^*\bm\omega^\dagger_J)$ with $V^n$ defined as in Section \ref{sec4-2} and let $\varepsilon_h=\max_{1\leq j\leq J}\|\bm{u}^\infty_j-T^\infty\bm{\omega}^\dagger_j\|$.
Therefore, for noisy measurement data $\bm u^{\infty,\delta}$,
we have that for $z^\dagger=(\bm\om^\dagger,\bm m^\dagger)$,
\be
\|\nabla F^{csi}(z^\dagger;\bm u^{\infty,\delta})\|_\infty &=&\max_{1\leq j\leq J}\left\|2\eta_{d,j}(T^\infty)^*(T^\infty \bm\om_j^\dagger-\bm u^{\infty,\delta}_j)\right\|_\infty\no\\
&\le& \max_{1\leq j\leq J}
\{2\eta_{d,j}\} \|(T^\infty)^*\|_{\infty,2}
(\varepsilon_h+\delta),\label{e5-20}
\en
where $\|\cdot\|_{\infty,2}$ is defined by $\|A\|_{\infty,2}:=\max_{i}(\sum_{j}|a_{ij}|^2)^{1/2}$
for any matrix $A=(a_{ij})$.
Moreover, we have that for $z^\dagger=(\bm\alpha^\dagger,\bm m^\dagger)$ and sufficiently small $h>0$,
\be
&~&\|\nabla F^{som}(z^\dagger;\bm u^{\infty,\delta})\|_\infty \notag\\
&=& \max \left\{ \max_{1\leq j\leq J} \|\nabla_{\bm \al_j} F^{som}(z^\dagger;\bm u^{\infty,\delta})\|_\infty, \|\nabla_{\bm m} F^{som}(z^\dagger;\bm u^{\infty,\delta})\|_\infty \right\}\no\\
&\leq& C_0 \max_{1\leq j\leq J} \left\{2\eta_{s,j} \|\bm\om^s_j+V^n\bm\al^\dagger_j-\bm\om^\dagger_j\|_\infty \right\} \no\\
&~& + \max_{1\leq j\leq J} \left\{2\eta_{d,j} \left\|(T^\infty V^n)^*(T^\infty(\bm\om_j^s+V^n\bm\al_j^\dagger) - \bm u_j^{\infty,\delta})\right\|_\infty \right\}\no\\
&\leq& C 
\max\left[
\|\widehat{V}^s\|_{\infty,2} \max_{1\leq j\leq J} \{2\eta_{s,j}\} , \la_{L_\al+1} \max_{1\leq j\leq J} \{2\eta_{d,j}\}\right](\varepsilon_h+\delta)\label{eq1}
\en
with
$\widehat{V}^s := (\bm v_1/\lambda_1,\dots,\bm v_{L_\al}/\lambda_{L_\al})\in\mathbb{C}^{M\times L_{\al}}$ and with $C_0$ and $C$ being constants independent of $h$.
Here, we use the facts that
$\limsup_{h\rightarrow +0}\|T\|_\infty<\infty$ (see \cite{V93,V2000}),
\ben
\|\bm\om^s_j+V^n\bm\al^\dagger_j-\bm\om^\dagger_j\|_\infty &=& \left\|\sum_{k=1}^{L_\al}\frac{\langle \bm u_j^{\infty,\delta} - T^\infty \bm\om_j^\dagger,\bm u_k\rangle}{\la_k} \bm v_k \right\|_\infty \\
&\leq& \|\widehat{V}^s\|_{\infty,2} \left\|T^\infty\bm\om_j^\dagger - \bm u_j^{\infty,\delta} \right\|
\enn
and
\begin{align*}
&\left\|(T^\infty V^n)^*(T^\infty(\bm\om_j^s+V^n\bm\al_j^\dagger) - \bm u^{\infty,\delta}_j)\right\|_\infty \\
=& \left\| (T^\infty V^n)^*\sum_{k\geq L_\al+1} \langle T^\infty\bm\om_j^\dagger - \bm u^{\infty,\delta}_j,\bm u_k\rangle \bm u_k \right\|_\infty 
\leq \la_{L_\al+1}
\left\|T^\infty\bm\om_j^\dagger - \bm u_j^{\infty,\delta} \right\|.
\end{align*}
It is known from \cite{V93,V2000} that 
$\lim_{h\rightarrow+0}\vep_h=0$.
Therefore, for sufficiently small $h>0$, the $z^\dagger$ satisfying (\ref{e5-20}) (resp. \eqref{eq1}) can be called the $O(\delta^{csi})$-stationary point of the objective function 
$F^{csi}(z;\bm u^{\infty,\delta})$ (resp. the $O(\delta^{som})$-stationary point of the objective function $F^{som}(z;\bm u^{\infty,\delta})$), where $\delta^{csi}$ and $\delta^{som}$ are given by
\begin{align}\label{eq3}
\left\{
\begin{aligned}
&\delta^{csi}:=
\|(T^\infty)^*\|_{\infty,2}
\max_{1\leq j\leq J}
\{2\eta_{d,j}\}
\delta,\\
&\delta^{som}:=\max\left[
\|\widehat{V}^s\|_{\infty,2} \max_{1\leq j\leq J} \{2\eta_{s,j}\} , \la_{L_\al+1} \max_{1\leq j\leq J} \{2\eta_{d,j}\}\right]\delta.
\end{aligned}
\right.
\end{align}
From (\ref{e5-18}), (\ref{e5-20}) and
\eqref{eq1}, we know that in the noisy data case, 
for sufficiently small $h>0$,
it is better to choose the regularization parameters 
$\gamma$ and $\beta$ such that
\be\label{e5-21}
&&\max\{\gamma L,\beta\} =
O(\delta^{csi})
\quad\mbox{for the IRCSI algorithm}\\
&&(\mbox{or}~\max\{\gamma L,\beta\} = 
O(\delta^{som})
 \quad\mbox{for the IRSOM algorithm}). \notag
\en
This is indeed confirmed by the numerical examples in the next section.

\section{Numerical experiments}\label{sec6}

This section presents several numerical examples to evaluate the performance of Algorithms \ref{alg1} and \ref{alg2}.
For the objective function $F^{csi}$ in \eqref{e3-1}, we set
\ben
\eta_{s,j}=\left(\sum_{j=1}^J\|\bm m^{(0)}\odot\bm u^i_j\|^2\right)^{-1}
\mathrm{~and~}
\eta_{d,j}=\left(\sum_{j=1}^J\|\bm u^{\infty,\delta}_j\|^2\right)^{-1},\quad j=1,\dots,J,
\enn
which are positive constants. Then Algorithm \ref{alg1} with $\gamma=\beta=0$ is just the original CSI method \cite{V1997} with the weights $\eta_{s,j}$ and $\eta_{d,j}$ set to be positive constants. For the objective function $F^{som}$ in \eqref{e3-4},
we set $\eta_{s,j}$ and $\eta_{d,j}$ as in \cite{C2010}
and choose $L_\al=10$. Then Algorithm \ref{alg2} with $\gamma=\beta=0$ is just the original SOM method \cite{C2010}.

We only consider the two-dimensional case in the numerical examples. In all the numerical examples,
we choose the scatterers located in the region $[-2,2]\x [-2,2]$. For the synthetic far-field data, we discretize the above region into $256\x256$ small boxes using the method presented in Section \ref{sec2}. To avoid the inverse crime, we discretize the same region into $64\x64$ small boxes via the same method in each iteration of our algorithms.
Furthermore, we choose $\gamma=\beta/64$ and $L=64$ so that $\gamma L=\beta$ always holds. Then $\varepsilon=\beta$ (see Theorem \ref{th1}) is the only varying parameter for different IRCSI-type algorithms.
We will compare the performance of five IRCSI algorithms (with $\beta=0,10^{-6},10^{-5},$ $10^{-4}$ and $10^{-3}$)
and five IRSOM algorithms (with $\beta=0,10^{-5},10^{-4},10^{-3}$ and $10^{-2}$).
To study the convergence property of these IRCSI-type algorithms, we allow these algorithms to run a
sufficiently large number of iterations without using the termination criterion \eqref{e5-17}.
In addition, in order to quantitatively evaluate the reconstruction results, we calculate the relative
errors (RE) at all iterations for each example. The relative error of the reconstruction result $y^k$
at the $k$-th iteration of the contrast is defined as
\ben
R^k:=\frac{\|y^k-y^\dagger\|_F}{\|y^\dagger\|_F},
\enn
where $y^\dagger$ is the ground truth contrast and $\|\cdot\|_F$ is the Frobenius norm.
We consider both noiseless and noisy measurement data. For the noisy case, the measurement data
$\bm u^{\infty,\delta}_j$ are obtained by adding $5\%$ relative noise to the exact far-field data $\bm u_j^\infty$, that is,
\ben
\frac{\|\bm u^{\infty,\delta}_j - \bm u^{\infty}_j\|}{\|\bm u^{\infty}_j\|}\le 5\%,\;\;\; j=1,\dots,J.
\enn
This means that $\bm{u}^{\infty,\delta}_j$ ($j=1,\ldots,J$) satisfy \eqref{eq2} with
\be\label{eq4}
\delta=0.05 \max_{1\leq j \leq J} \|\bm u_j^{\infty}\|\approx 0.05 \max_{1\leq j \leq J} \|\bm u_j^{\infty,\delta}\|.
\en
In the following two subsections, we present numerical examples for the case of a single scatterer and the case of multiple scatterers, respectively.

\subsection{Single scatterer}
In the first example, we consider a scatterer whose contrast $m(x)$ is $e^{-1/(1-|x|^2)}$ for $|x|<1$ and zero elsewhere. 
The wave number $\ka=6$ and $J=Q=16$.
We choose $d_j=(\cos(2\pi (j-1)/J),\sin(2\pi (j-1)/J))$ ($j=1,\dots,J$)
and $\hat{x}^q=(\cos(2\pi (q-1)/Q),\sin(2\pi (q-1)/Q))$ ($q=1,\dots,Q$), which are $16$ incident and $16$ measurement
directions uniformly distributed on the unit circle.
Fig. \ref{f1_noise} and Fig. \ref{f1_noiseless} present the ground truth and the reconstruction results of
the IRCSI algorithm with $\beta=0,10^{-6},10^{-5},10^{-4},10^{-3}$ and the IRSOM algorithm with
$\beta=0,10^{-5},10^{-4},10^{-3},10^{-2}$ both at the $30000$-th iteration from
the noisy data with $5\%$ noise and at the 5000-th iteration from the noiseless data, respectively.
The top row in Fig. \ref{f1_noise} and Fig. \ref{f1_noiseless} shows the ground truth and the reconstruction
results of the IRCSI algorithm, whilst the bottom row in Fig. \ref{f1_noise} and Fig. \ref{f1_noiseless}
shows the reconstruction results of the IRSOM algorithm.
Fig. \ref{f2_noise} and Fig. \ref{f2_noiseless} present the relative error (RE) $R^k$ of the reconstruction result
$y^k$ at the $k$-th iteration against the iteration step $k$ in the cases of noisy and noiseless data, respectively.

\begin{figure}[htbp]
\centering
\includegraphics[width=13cm,height=4.2cm]{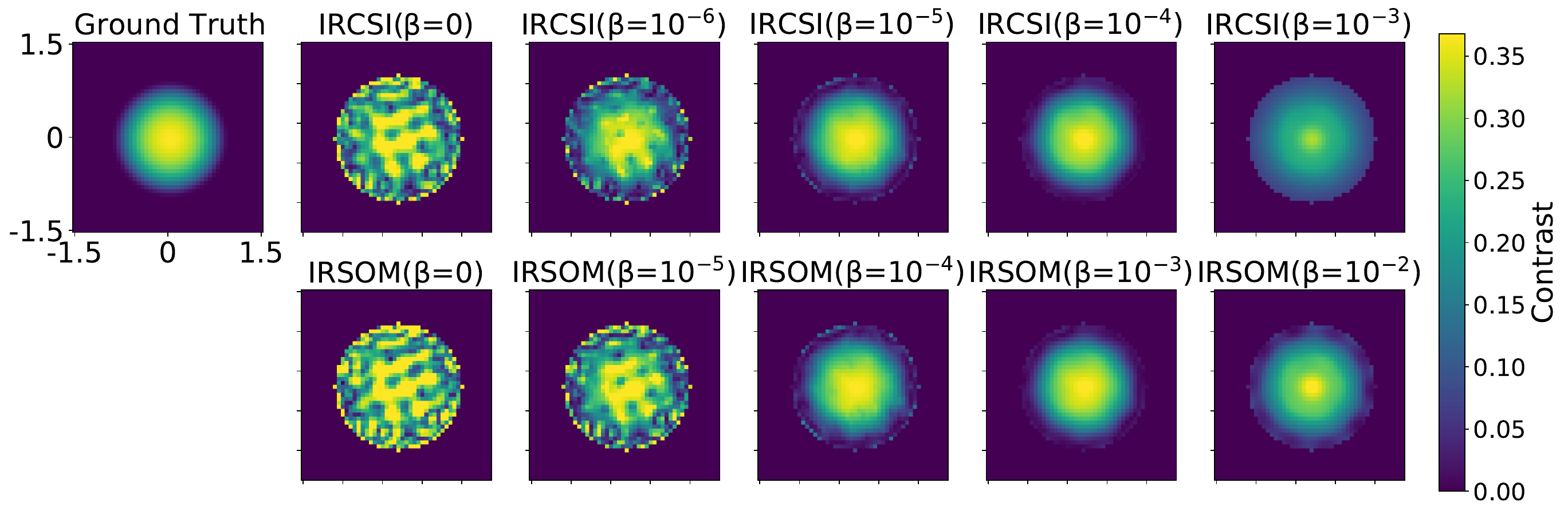}
\caption{The ground truth $m(x)=e^{-1/(1-|x|^2)}$ and the reconstruction results by the IRCSI and IRSOM algorithms
with different values of $\beta$ at the $30000$-th iteration from the noisy measurement data with $5\%$ noise.
Top row from left to right: the ground truth and the reconstructions by the IRCSI algorithm with
$\beta=0,10^{-6},10^{-5},10^{-4},10^{-3}$. Bottom row from left to right: the reconstructions by the IRSOM algorithm with
$\beta=0,10^{-5},10^{-4},10^{-3},10^{-2}$.
}\label{f1_noise}
\end{figure}

\begin{figure}[htbp]
\centering
\includegraphics[width=13cm,height=4.2cm]{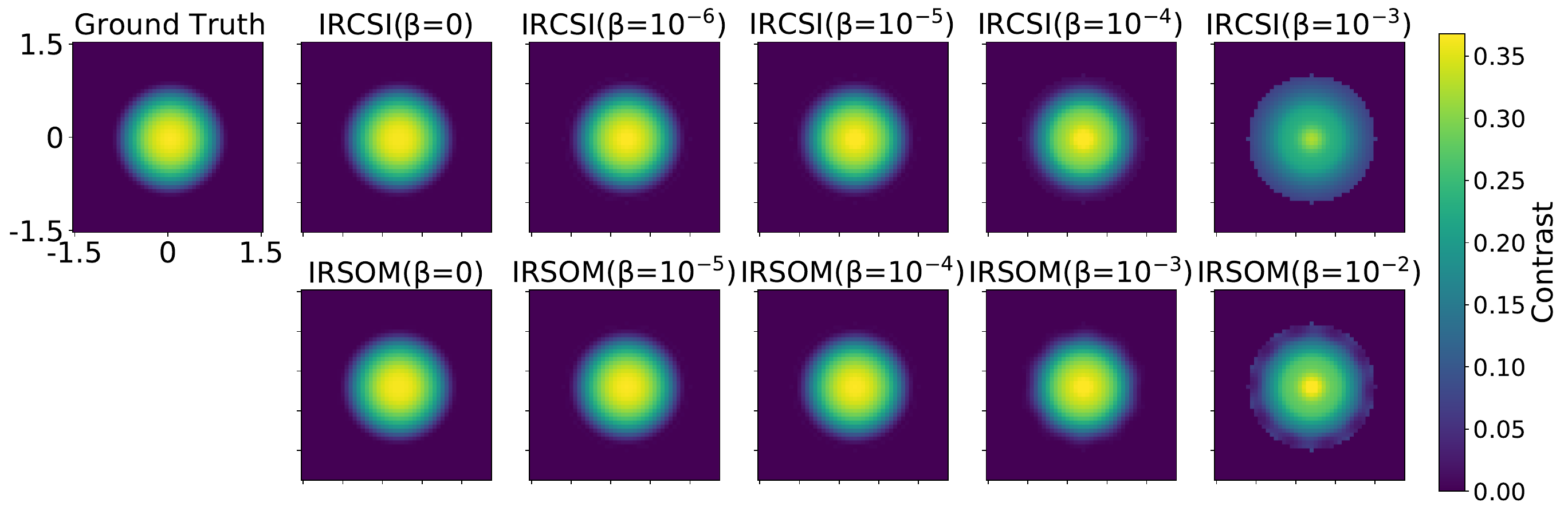}
\caption{The ground truth $m(x)=e^{-1/(1-|x|^2)}$ and the reconstruction results by the IRCSI and IRSOM algorithms
with different values of $\beta$ at the $5000$-th iteration from the noiseless measurement data.
Top row from left to right: the ground truth and the reconstructions by the IRCSI algorithm with
$\beta=0,10^{-6},10^{-5},10^{-4},10^{-3}$. Bottom row from left to right: the reconstructions by the IRSOM algorithm with
$\beta=0,10^{-5},10^{-4},10^{-3},10^{-2}$.
}\label{f1_noiseless}
\end{figure}

\begin{figure}
\begin{minipage}[t]{0.49\linewidth}
\includegraphics[width=6.3cm,height=4.5cm]{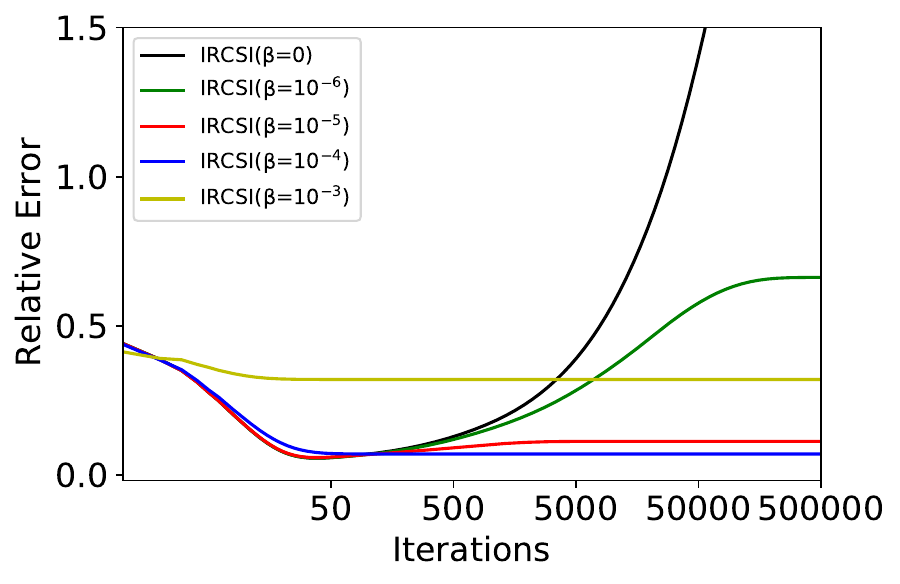}
\end{minipage}
\begin{minipage}[t]{0.49\linewidth}
\includegraphics[width=6.3cm,height=4.5cm]{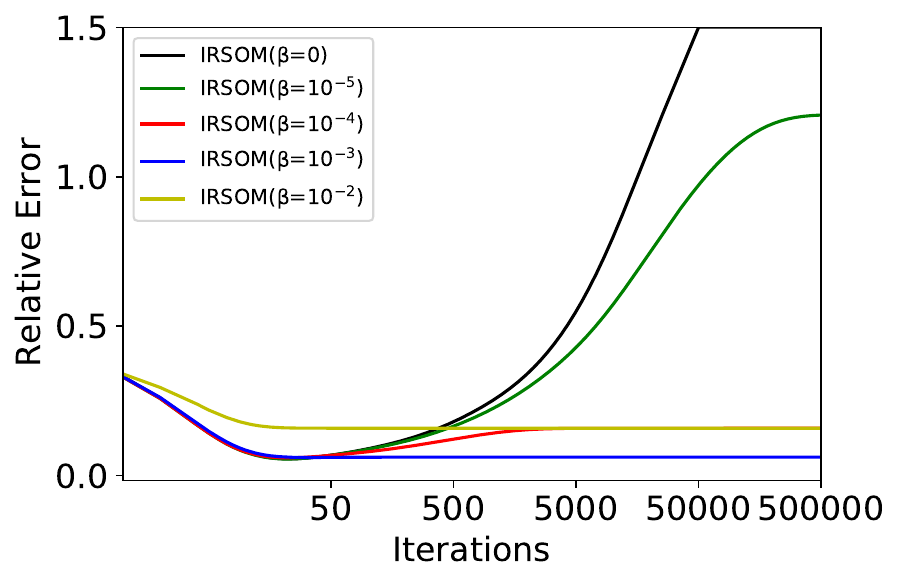}
\end{minipage}
\caption{The relative error between the ground truth and the reconstruction results by the IRCSI algorithm
with $\beta=0,10^{-6},10^{-5},10^{-4},10^{-3}$ (left figure) and the IRSOM algorithm with
$\beta=0,10^{-5},10^{-4},10^{-3},10^{-2}$ (right figure) against the iteration step, where
the ground truth contrast is $m(x)=e^{-1/(1-|x|^2)}$ and the relative noise of the measurement data is $5\%$.
}\label{f2_noise}
\end{figure}

\begin{figure}[htbp]
\begin{minipage}[t]{0.49\linewidth}
\includegraphics[width=6.3cm,height=4.5cm]{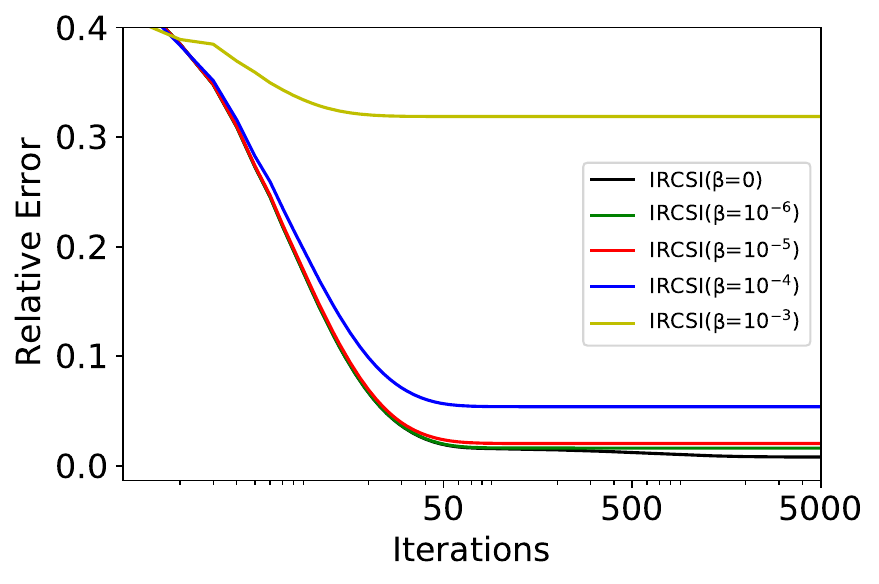}
\end{minipage}
\begin{minipage}[t]{0.49\linewidth}
\includegraphics[width=6.3cm,height=4.5cm]{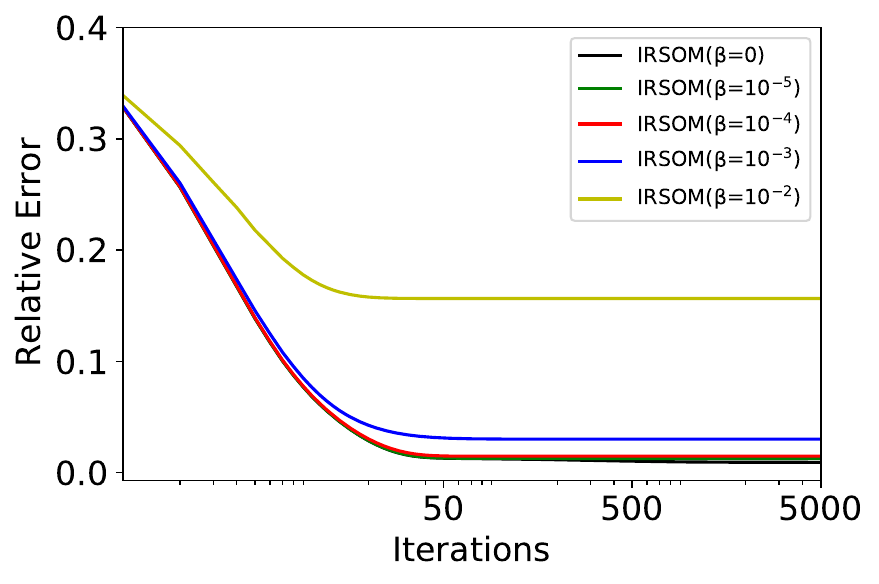}
\end{minipage}
\caption{The relative error between the ground truth and the reconstruction results by the IRCSI algorithm
with $\beta=0,10^{-6},10^{-5},10^{-4},10^{-3}$ (left figure) and the IRSOM algorithm with
$\beta=0,10^{-5},10^{-4},10^{-3},10^{-2}$ (right figure) against the iteration step, where
the ground truth contrast is $m(x)=e^{-1/(1-|x|^2)}$ and the measurement data has no noise.
}\label{f2_noiseless}
\end{figure}

In the second example, we consider a scatterer for which
the shape of the support of its contrast $m(x)$ is the handwritten digit $0$ from the MNIST dataset \cite{D2012}.
The handwritten digit $0$ is a grayscale image, where the graysacle value of each pixel ranges from an integer
in $[0,255]$, and the contrast value of the scatterer, ranging from $0$ to $1$, is obtained by dividing the
grayscale value by $255$.
The other settings and the parameters of this example are the same as in the first example.
Fig. \ref{f4_noise} and Fig. \ref{f4_noiseless} present the ground truth and the reconstruction results of
the IRCSI algorithm with $\beta=0,10^{-6},10^{-5},10^{-4},10^{-3}$ and the IRSOM algorithm with
$\beta=0,10^{-5},10^{-4},10^{-3},10^{-2}$ both at the $50000$-th iteration from
the noisy data with $5\%$ noise and at the 5000-th iteration from the noiseless data, respectively.
The top row in Fig. \ref{f4_noise} and Fig. \ref{f4_noiseless} shows the ground truth and the reconstruction
results of the IRCSI algorithm, whilst the bottom row in Fig. \ref{f4_noise} and Fig. \ref{f4_noiseless} gives
the reconstruction results of the IRSOM algorithm.
Fig. \ref{f5_noise} and Fig. \ref{f5_noiseless} show the relative error (RE) $R^k$ of the reconstruction result
$y^k$ at the $k$-th iteration against the iteration step $k$ in the cases of noisy and noiseless data, respectively.

\begin{figure}[htbp]
\centering
\includegraphics[width=13cm,height=4.2cm]{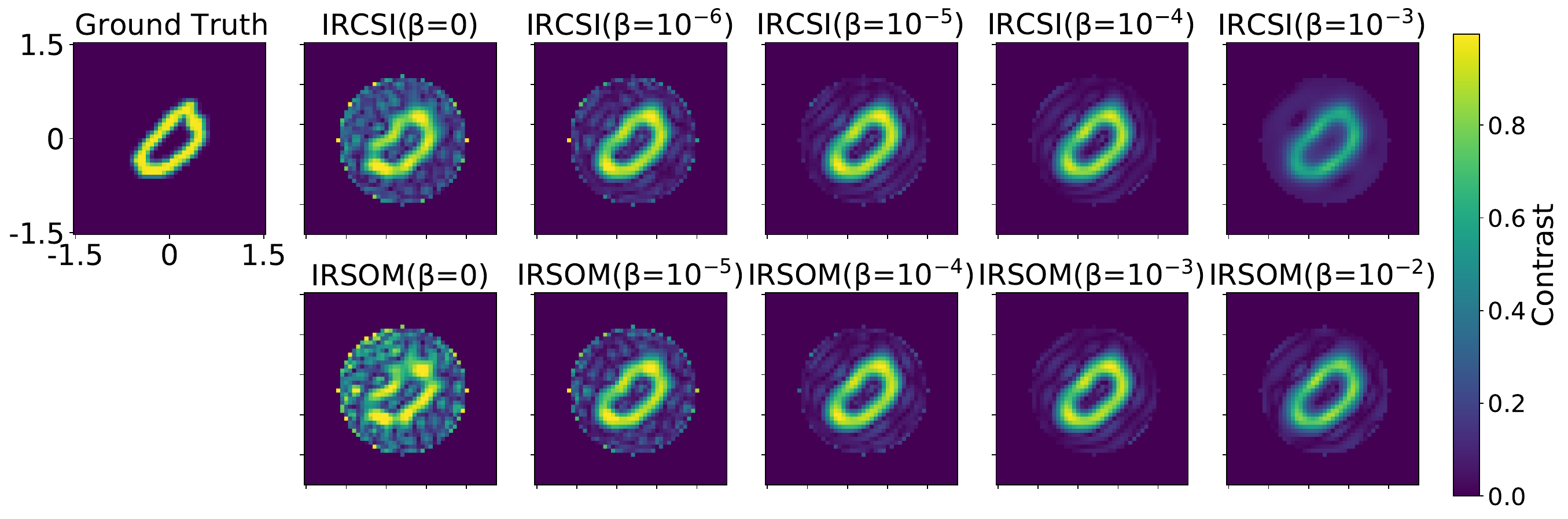}
\caption{The ground truth of the handwritten digit $0$ image and the reconstructions by the IRCSI and IRSOM
algorithms with different values of $\beta$ at the $50000$-th iteration from the noisy data with $5\%$ noise.
Top row from left to right: the ground truth and the reconstructions by the IRCSI algorithm with
$\beta=0,10^{-6},10^{-5},10^{-4},10^{-3}$. Bottom row from left to right: the reconstructions by the IRSOM algorithm with
$\beta=0,10^{-5},10^{-4},10^{-3},10^{-2}$.
}\label{f4_noise}
\end{figure}

\begin{figure}[htbp]
\centering
\includegraphics[width=13cm,height=4.2cm]{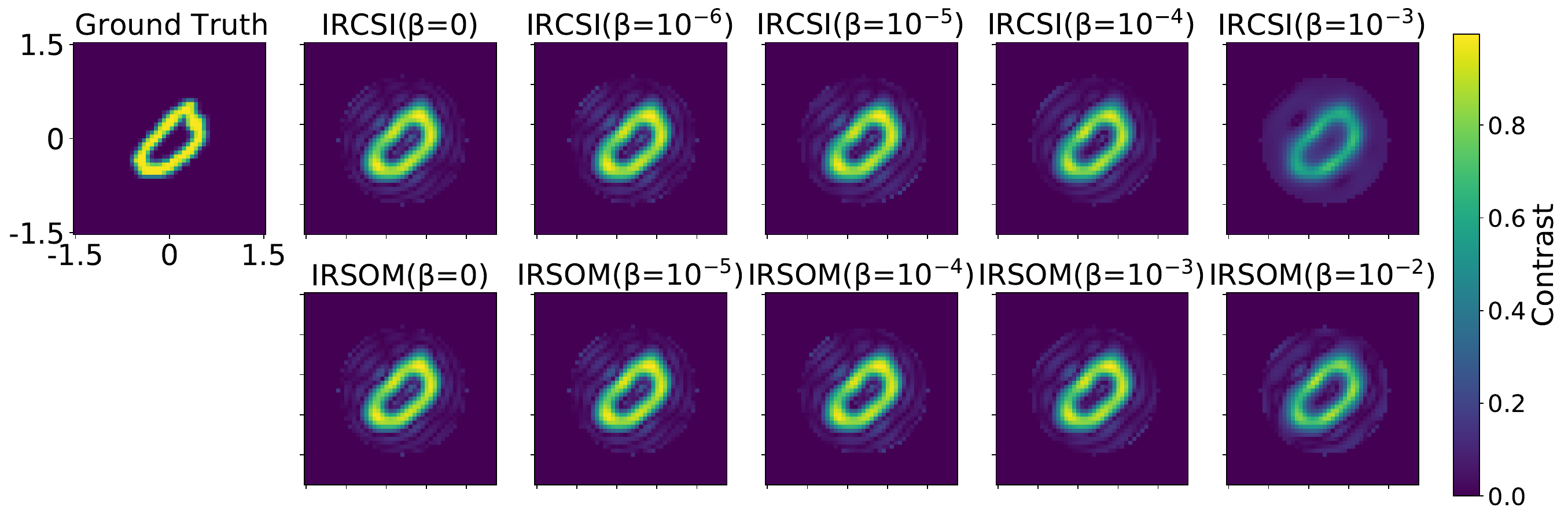}
\caption{The ground truth of the handwritten digit $0$ image and the reconstructions by the IRCSI and IRSOM
algorithms with different values of $\beta$ at the $5000$-th iteration from the noiseless data.
Top row from left to right: the ground truth and the reconstructions by the IRCSI algorithm with
$\beta=0,10^{-6},10^{-5},10^{-4},10^{-3}$. Bottom row from left to right: the reconstructions by the IRSOM algorithm with
$\beta=0,10^{-5},10^{-4},10^{-3},10^{-2}$.
}\label{f4_noiseless}
\end{figure}

\begin{figure}[htbp]
\begin{minipage}[t]{0.49\linewidth}
\includegraphics[width=6.3cm,height=4.5cm]{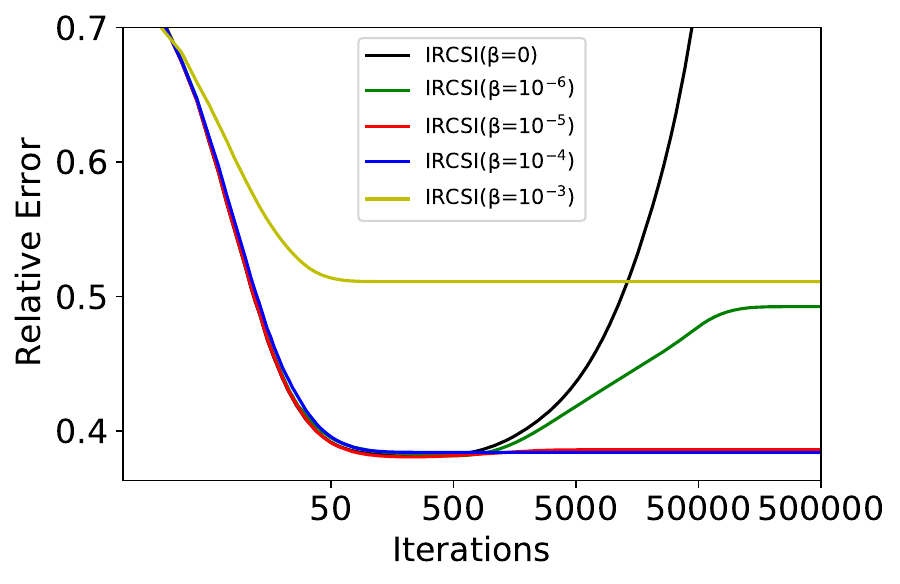}
\end{minipage}
\begin{minipage}[t]{0.49\linewidth}
\includegraphics[width=6.3cm,height=4.5cm]{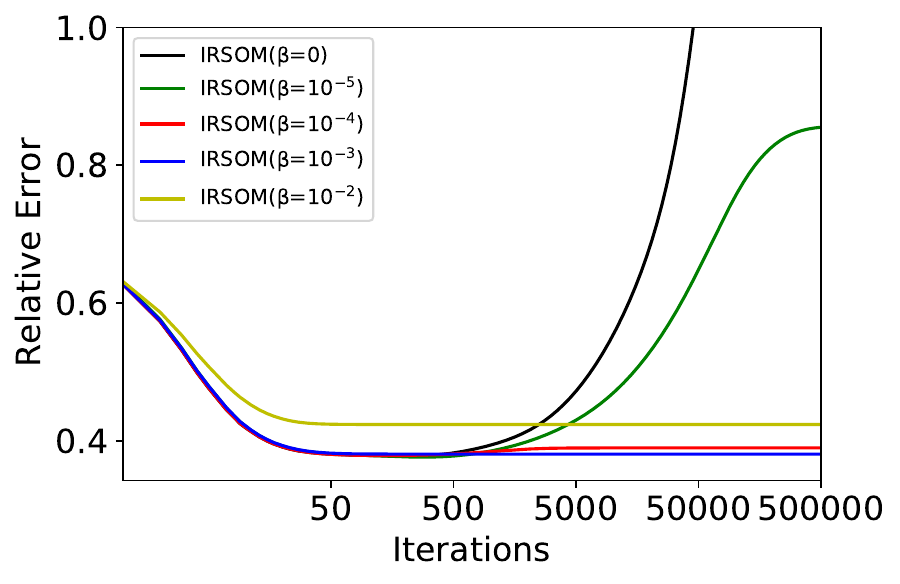}
\end{minipage}
\caption{The relative error between the ground truth and the reconstruction results by the IRCSI algorithm
with $\beta=0,10^{-6},10^{-5},10^{-4},10^{-3}$ (left figure) and the IRSOM algorithm with
$\beta=0,10^{-5},10^{-4},10^{-3},10^{-2}$ (right figure) against the iteration step for the handwritten digit $0$
image case, where the relative noise of the measurement data is $5\%$.
}\label{f5_noise}
\end{figure}

\begin{figure}[htbp]
\begin{minipage}[t]{0.49\linewidth}
\includegraphics[width=6.3cm,height=4.5cm]{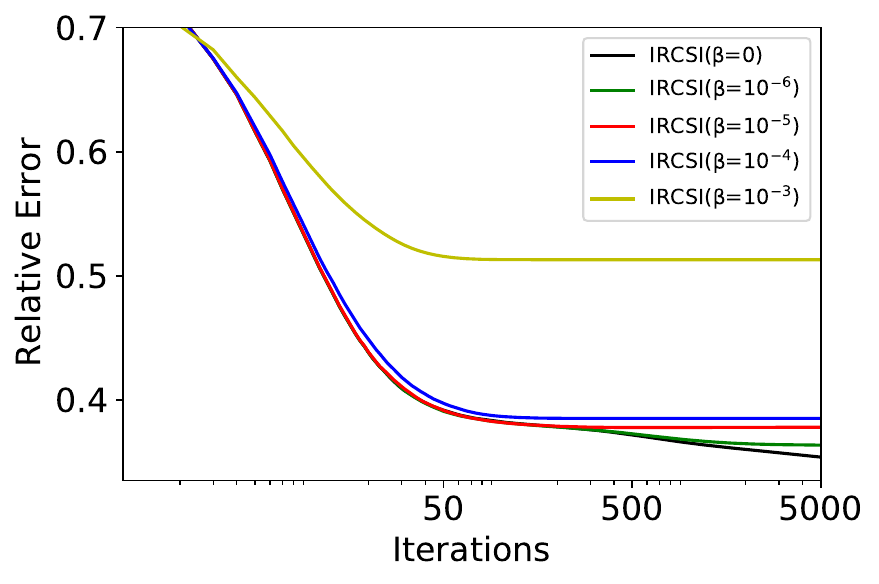}
\end{minipage}
\begin{minipage}[t]{0.49\linewidth}
\includegraphics[width=6.3cm,height=4.5cm]{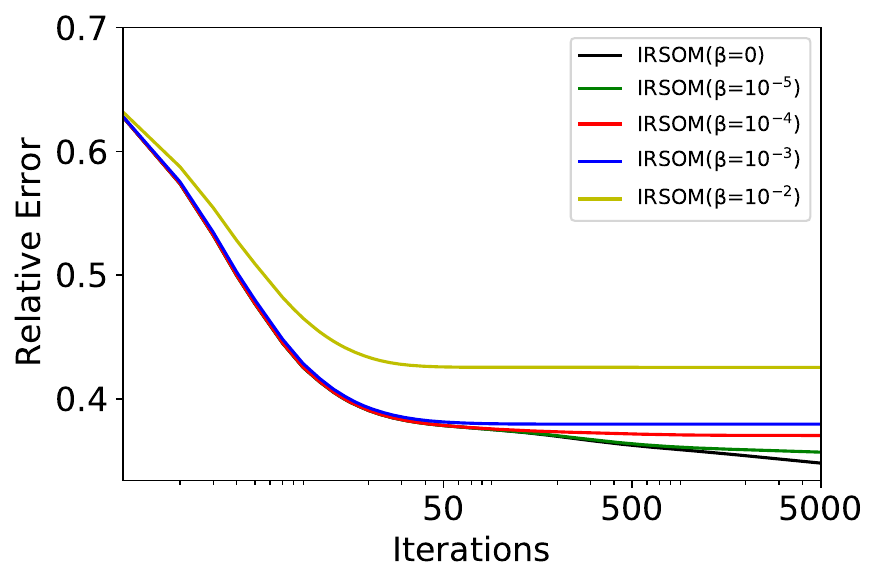}
\end{minipage}
\caption{The relative error between the ground truth and the reconstruction results by the IRCSI algorithm
with $\beta=0,10^{-6},10^{-5},10^{-4},10^{-3}$ (left figure) and the IRSOM algorithm with
$\beta=0,10^{-5},10^{-4},10^{-3},10^{-2}$ (right figure) against the iteration step in the handwritten
digit $0$ image case with noiseless data.
}\label{f5_noiseless}
\end{figure}

From the numerical results presented above we have the following observations.

1) In the noiseless data case, the original CSI and SOM algorithms (that is, the IRCSI and IRSOM algorithms
with $\beta=0$, respectively) and the IRCSI and IRSOM algorithms with small $\beta>0$ can give satisfactory
reconstructions (see Fig. \ref{f1_noiseless} and Fig. \ref{f4_noiseless} for the IRCSI algorithm with $\beta=10^{-6},10^{-5},10^{-4}$ and for the IRSOM algorithm with $\beta=10^{-5},10^{-4},10^{-3}$).
Moreover, from the relative error of the reconstruction result at each iteration
(see Fig. \ref{f2_noiseless} and Fig. \ref{f5_noiseless}) it is found that the original CSI and SOM algorithms are convergent in the noiseless data case.

2) In the noisy data case (see Fig.~\ref{f2_noise} and Fig.~\ref{f5_noise}), it seems that the original CSI and SOM algorithms may not be convergent since the relative errors obtained by these algorithms are still increasing after $500000$ iterations in both the first and second examples (though not fully presented in Fig.~\ref{f2_noise} and Fig.~\ref{f5_noise} due to vertical axis truncation).
However, the IRCSI and IRSOM algorithms with $\beta>0$ are all convergent in this case, as confirmed
by the theoretical convergence results presented in Corollaries \ref{c1} and \ref{c2}.
In particular, from Fig. \ref{f2_noise}, Fig. \ref{f5_noise}, it can be seen that for both two examples, the IRCSI algorithm with $\beta=10^{-6}$ and the IRSOM algorithm with $\beta=10^{-5}$ converge very slowly, but they all eventually converge after 500000 iterations.
Furthermore, Fig. \ref{f1_noise}, Fig. \ref{f2_noise}, Fig. \ref{f4_noise} and Fig. \ref{f5_noise} show that the reconstruction results obtained
by the original CSI and SOM algorithms are much worse than those obtained by the IRCSI and IRSOM algorithms
with $\beta>0$ in the noisy data case. 
In the noisy data case, the ground truth solution is a $O(\delta^{csi})$-stationary point of $F^{csi}(z;\bm u^{\infty,\delta})$ or a $O(\delta^{som})$-stationary point of $F^{som}(z;\bm u^{\infty,\delta})$ (see (\ref{e5-20}), \eqref{eq1} and \eqref{eq3}), 
but the limit point $z^*$ of the iterative sequence obtained by the IRCSI and IRSOM algorithms is a 
$\beta$-stationary point of the objective function $F^{csi}(z;\bm u^{\infty,\delta})$ or 
$F^{som}(z;\bm u^{\infty,\delta})$ (see (\ref{e5-18})) (noting that $\vep=\beta$ here). Thus, if $\beta$ is 
much smaller than $\delta^{csi}$ for the IRCSI algorithm (or $\delta^{som}$ for the IRSOM algorithm), then the limit point $z^*$ would be very different from the ground truth 
solution, as confirmed by Fig. \ref{f1_noise}, Fig. \ref{f2_noise}, Fig. \ref{f4_noise} and Fig. \ref{f5_noise}.
Moreover, it can also be observed from the four above figures that the IRCSI algorithm with $\beta=10^{-4}$ and the IRSOM algorithm with $\beta=10^{-3}$ converge within only several dozen iterative steps and provide satisfactory reconstruction results. This is indeed consistent with the selection strategy of $\beta$ given in \eqref{e5-21}, since we can use \eqref{eq3}, \eqref{eq4} and the measurement data
$\bm u^{\infty,\delta}_j$ to calculate that
\ben
&&\delta^{csi}\approx 1.610\x10^{-4}~\mbox{and}~\delta^{som}\approx 8.607\x10^{-3}\quad\mbox{in Example 1},\\
&&\delta^{csi}\approx 1.365\x10^{-4}~\mbox{and}~\delta^{som}\approx 9.440\x10^{-3}\quad\mbox{in Example 2}.
\enn

\subsection{Multiple scatterers}\label{sec:multiple}
In the third example, we consider the Austria model, which consists of three separated scatterers: two circular disks and one circular ring, each with a different constant contrast value; see the ground truth in Fig.~\ref{f6_noise} and Fig.~\ref{f6_noiseless}. For this model, the contrast values of the left disk, right disk and central ring are set to be  $1$, $1.5$ and $2$, respectively. All other settings and parameters remain the same as in the above two examples.
Fig.~\ref{f6_noise} and Fig.~\ref{f6_noiseless} present the ground truth and the reconstruction results of the IRCSI algorithm with $\beta=0,10^{-6},10^{-5},10^{-4},10^{-3}$ and the IRSOM algorithm with $\beta=0,10^{-5},10^{-4},10^{-3},10^{-2}$ both at the $50000$-th iteration from the noisy data with $5\%$ noise and at the $5000$-th iteration from the noiseless data, respectively. Fig.~\ref{f7_noise} and Fig.~\ref{f7_noiseless} show the relative error $R^k$ of the reconstruction result $y^k$ against the iteration step $k$ in the noisy and noiseless cases, respectively.
It can be observed from Figs. \ref{f6_noise}, \ref{f6_noiseless}, \ref{f7_noise} and \ref{f7_noiseless} that the performance of the IRCSI and IRSOM algorithms in this multiple-scatterer example is similar to that in the single-scatterer examples of the previous subsection.
In particular, 
it can be seen from Fig. \ref{f6_noiseless} that in the noiseless case, the original CSI and SOM algorithms (that is, the IRCSI and IRSOM algorithms
with $\beta=0$, respectively) and the IRCSI and IRSOM algorithms with sufficiently small $\beta>0$ can provide satisfactory reconstructions. It is shown in Fig. \ref{f7_noise} that in the noisy case, the IRCSI and IRSOM algorithms with $\beta>0$ are all convergent.
Moreover, in the noisy case, it can be observed from Figs. \ref{f6_noise} and \ref{f7_noise} that the IRCSI algorithm with $\beta=10^{-5}$ and the IRSOM algorithm with $\beta=10^{-4}$ converge within only several hundred iterative steps and provide satisfactory reconstruction results. This is indeed consistent with the selection strategy of $\beta$ given in \eqref{e5-21}, since we can also use \eqref{eq3}, \eqref{eq4} and the measurement data
$\bm u^{\infty,\delta}_j$ to calculate that
\ben
\delta^{csi}\approx 8.715\x10^{-5} ~\mbox{and}~\delta^{som}\approx 4.952\x10^{-3}\quad\mbox{in Example 3.}
\enn

\begin{figure}[htbp]
\centering
\includegraphics[width=13cm,height=4.2cm]{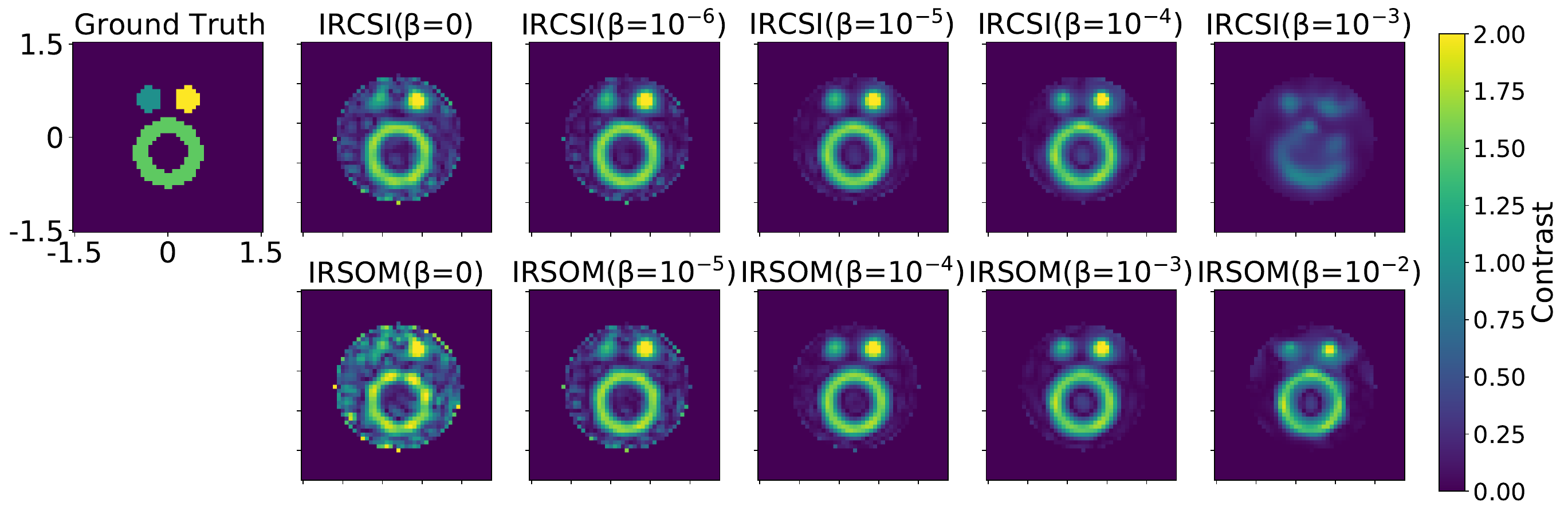}
\caption{The ground truth of the Austria model and the reconstructions by the IRCSI and IRSOM
algorithms with different values of $\beta$ at the $50000$-th iteration from the noisy data with $5\%$ noise.
Top row from left to right: the ground truth and the reconstructions by the IRCSI algorithm with
$\beta=0,10^{-6},10^{-5},10^{-4},10^{-3}$. Bottom row from left to right: the reconstructions by the IRSOM algorithm with
$\beta=0,10^{-5},10^{-4},10^{-3},10^{-2}$.
}\label{f6_noise}
\end{figure}

\begin{figure}[htbp]
\centering
\includegraphics[width=13cm,height=4.2cm]{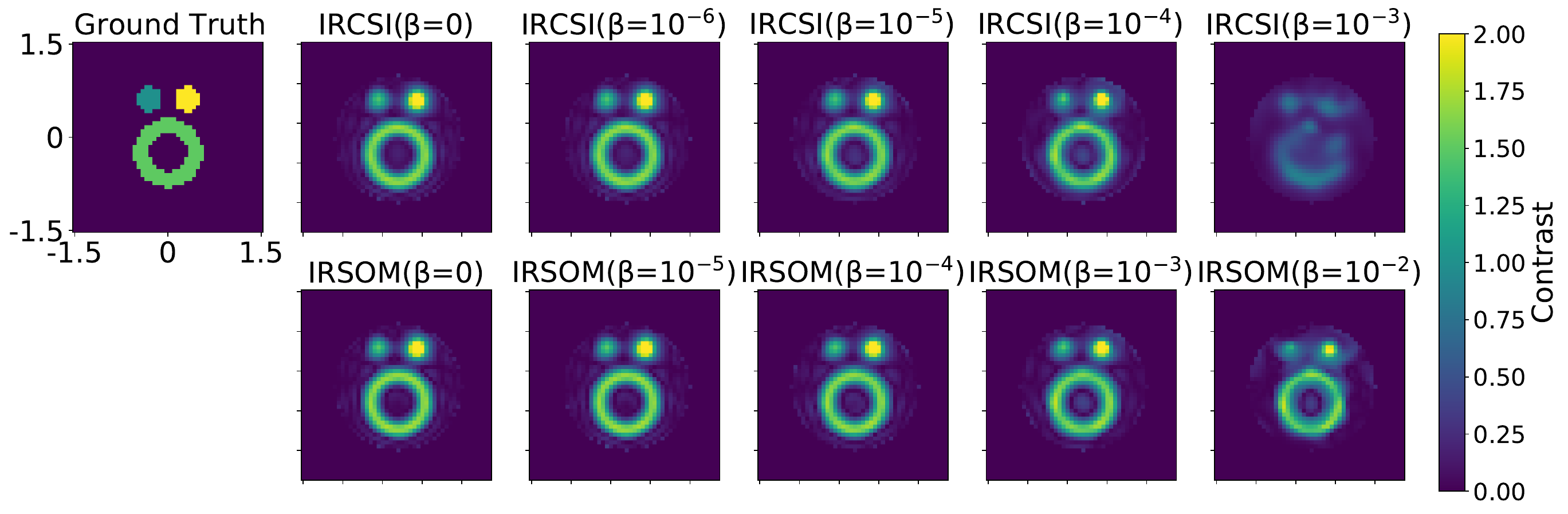}
\caption{The ground truth of the Austria model and the reconstructions by the IRCSI and IRSOM
algorithms with different values of $\beta$ at the $5000$-th iteration from the noiseless data.
Top row from left to right: the ground truth and the reconstructions by the IRCSI algorithm with
$\beta=0,10^{-6},10^{-5},10^{-4},10^{-3}$. Bottom row from left to right: the reconstructions by the IRSOM algorithm with
$\beta=0,10^{-5},10^{-4},10^{-3},10^{-2}$.
}\label{f6_noiseless}
\end{figure}

\begin{figure}[htbp]
\begin{minipage}[t]{0.49\linewidth}
\includegraphics[width=6.3cm,height=4.5cm]{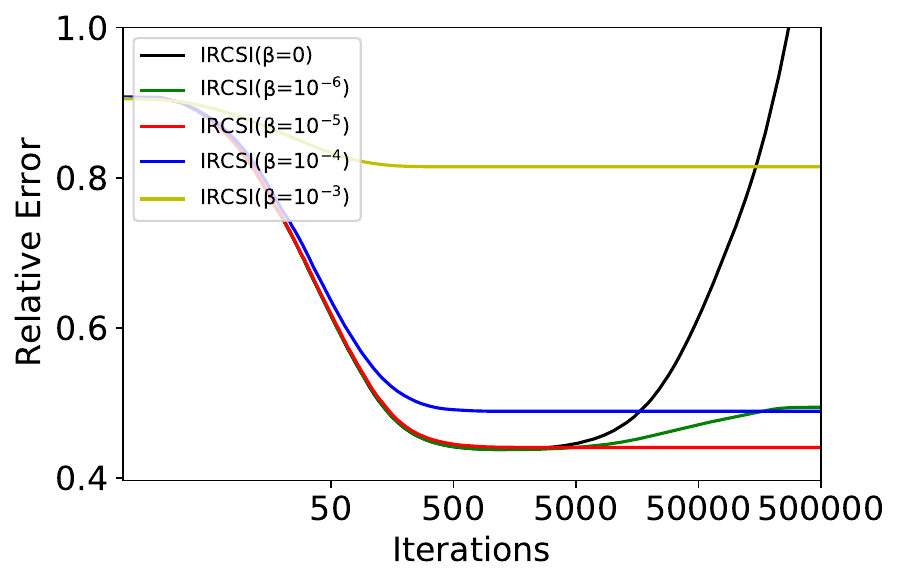}
\end{minipage}
\begin{minipage}[t]{0.49\linewidth}
\includegraphics[width=6.3cm,height=4.5cm]{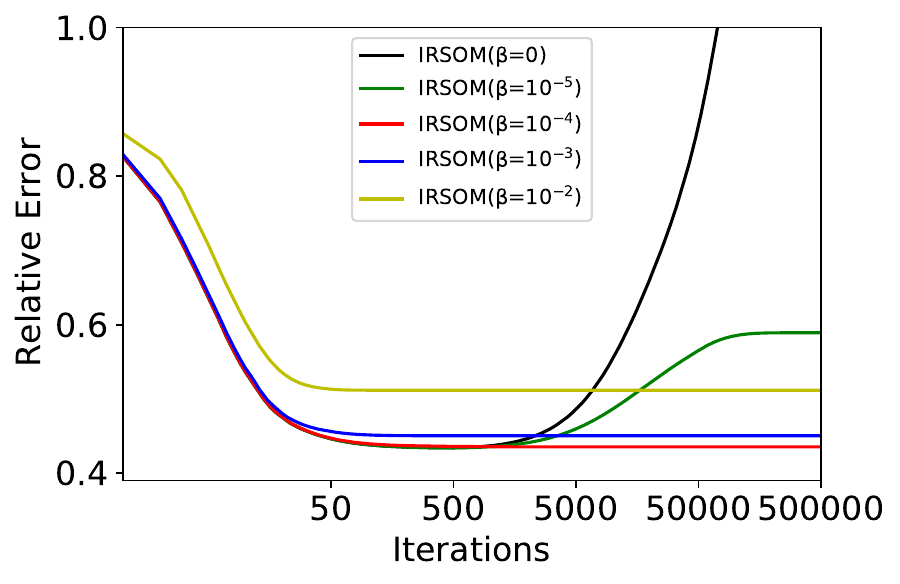}
\end{minipage}
\caption{The relative error between the ground truth and the reconstruction results by the IRCSI algorithm
with $\beta=0,10^{-6},10^{-5},10^{-4},10^{-3}$ (left figure) and the IRSOM algorithm with
$\beta=0,10^{-5},10^{-4},10^{-3},10^{-2}$ (right figure) against the iteration step for the Austria model, where the relative noise of the measurement data is $5\%$.
}\label{f7_noise}
\end{figure}

\begin{figure}[htbp]
\begin{minipage}[t]{0.49\linewidth}
\includegraphics[width=6.3cm,height=4.5cm]{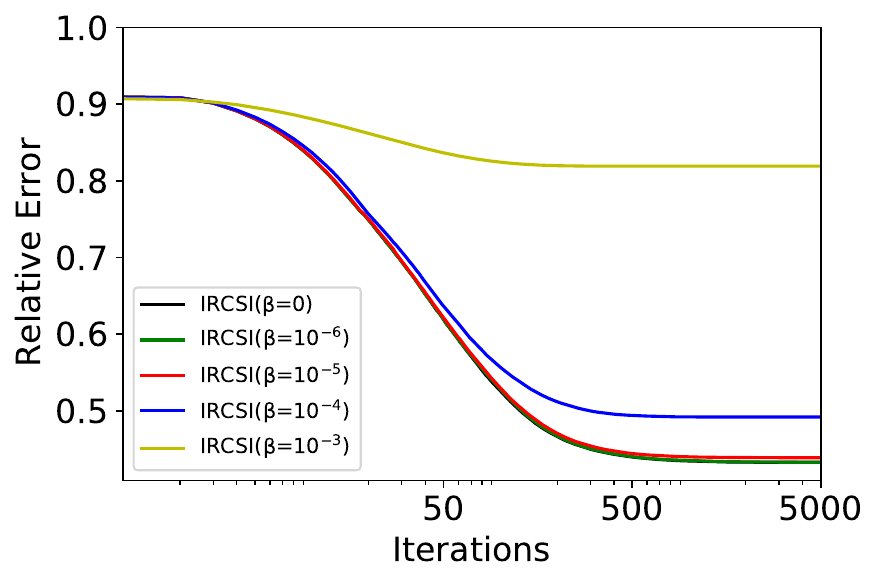}
\end{minipage}
\begin{minipage}[t]{0.49\linewidth}
\includegraphics[width=6.3cm,height=4.5cm]{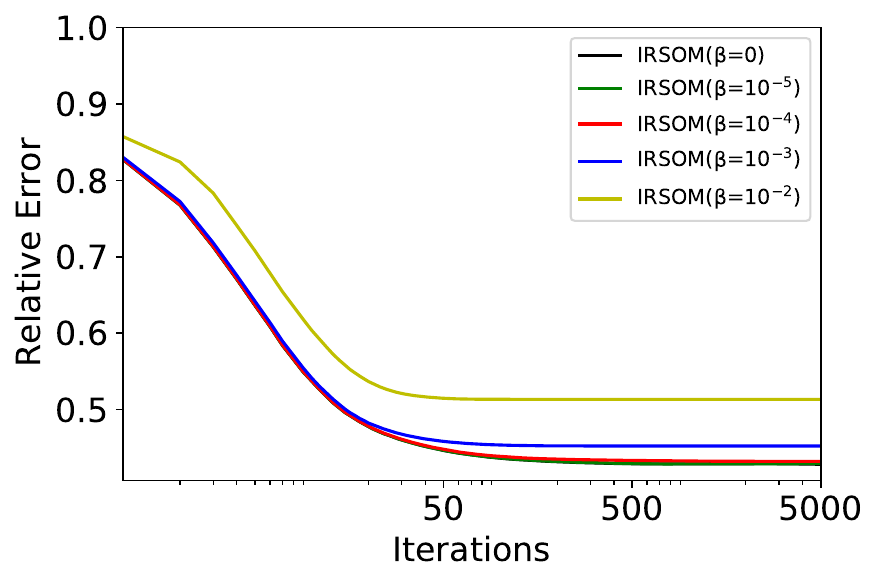}
\end{minipage}
\caption{The relative error between the ground truth and the reconstruction results by the IRCSI algorithm
with $\beta=0,10^{-6},10^{-5},10^{-4},10^{-3}$ (left figure) and the IRSOM algorithm with
$\beta=0,10^{-5},10^{-4},10^{-3},10^{-2}$ (right figure) against the iteration step for the Austria model with noiseless data.
}\label{f7_noiseless}
\end{figure}

\section{Conclusion}\label{sec7}

In this paper, we proposed two iteratively regularized CSI-type methods with a novel $\ell_1$ proximal term as 
the iteratively regularized term, the IRCSI and IRSOM algorithms, for solving inverse medium scattering problems
with a fixed frequency.
In the case when the regularization 
parameters $\gamma=\beta=0$, the IRCSI algorithm reduces to 
the original CSI algorithm with the weights $\eta_{s,j}$ and $\eta_{d,j}$ taking to be positive constants and the IRSOM algorithm reduces to the original SOM algorithm.
The IRCSI and IRSOM algorithms have a similar computational complexity to the original CSI and SOM algorithms, respectively. 
It was proved rigorously that the IRCSI and IRSOM algorithms are globally convergent  
under natural and weak conditions on the original objective function (see Corollaries \ref{c1} and \ref{c2}).
As far as we know, this is the first convergence result for iterative methods of solving nonlinear inverse 
scattering problems with a fixed frequency.
Numerical results are presented to demonstrate the theoretical convergence of the IRCSI and IRSOM algorithms 
as well as their much better performance in comparison with the original CSI and SOM algorithms, especially 
in the noisy data case.
It was also proved in the noisy data case that for sufficiently small $h>0$, the ground truth solution is a $O(\delta^{csi})$-stationary point of $F^{csi}(z;\bm u^{\infty,\delta})$ or a $O(\delta^{som})$-stationary point of $F^{som}(z;\bm u^{\infty,\delta})$ (see (\ref{e5-20}), \eqref{eq1} and \eqref{eq3}),
whilst the limit point $z^*$ of the iterative sequence obtained by the IRCSI and IRSOM algorithms is an
$\vep$-stationary point of the objective function $F^{csi}(z;\bm u^{\infty,\delta})$ or
$F^{som}(z;\bm u^{\infty,\delta})$ (see (\ref{e5-18})).
This suggested that in the noisy data case, it is better to choose the regularization parameters $\gamma$ and $\beta$ in the IRCSI and IRSOM algorithms according to the selection strategy given in \eqref{e5-21}, which was also confirmed by the numerical results.
In the future, we hope to provide a rigorous analysis on the choice of regularization parameters, investigate the convergence rates of the iterative sequences generated by our methods,
and study the error estimates between their limit points and the exact solution with respect to the noise level.

\section*{Acknowledgments}

This work was partially supported by
the National Key R\&D Program of China (2024YFA1012303),
the NNSF of China (12431016, 12271515)
and Youth Innovation Promotion Association CAS.

\end{document}